\documentclass[11pt]{amsart}

\usepackage{geometry}

\usepackage{colortbl}
\usepackage{amsmath,amssymb,amsthm}
\usepackage{mathrsfs}
\usepackage{graphicx}
\usepackage{dsfont}
\usepackage{multirow}
\usepackage{graphicx}
\usepackage{epstopdf}
\usepackage{color}
\usepackage{ulem}
\usepackage{hyperref}
\usepackage{enumerate}
\usepackage{caption}
\usepackage{subfig}
\usepackage{bm}
\usepackage{overpic}

\usepackage{tikz}

\usepackage{varioref}
\usepackage{hyperref}
\usepackage{cleveref}

\usepackage{algorithm}
\usepackage{algorithmic}

\graphicspath{{figures/}}

\makeatletter

\newcommand{\Rmnum}[1]{\expandafter\@slowromancap\romannumeral #1@}
\makeatother

\theoremstyle{definition}

\newtheorem{theorem}{Theorem}[section]
\newtheorem{lemma}{Lemma}[section]
\newtheorem{proposition}{Proposition}[section]
\newtheorem{remark}{Remark}[section]
\newtheorem{example}{Example}[section]
\newtheorem{assumption}{Assumption}[section]

\DeclareMathOperator*{\argmin}{arg\,min}
\DeclareMathOperator*{\arginf}{arg\,inf}

\numberwithin{equation}{section}

\begin{document}

\title[]{A model reduction method for solving the eigenvalue problem of semiclassical random Schr\"odinger operators}

\author{Panchi Li}
\address{Department of Mathematics, The University of Hong Kong, Hong Kong}
\email{lipch@hku.hk}
\author{Zhiwen Zhang$^{\ast}$}\thanks{*Corresponding author}
\address{Department of Mathematics, The University of Hong Kong, Hong Kong.
 Materials Innovation Institute for Life Sciences and Energy (MILES), HKU-SIRI, Shenzhen, P.R. China.}
\email[Corresponding author]{zhangzw@hku.hk}

\date{\today}

\dedicatory{}
\keywords{Eigenvalue problem, Semiclassical random Schr\"odinger operator, Proper orthogonal decomposition, multiscale model reduction method, convergence analysis.}

\subjclass[2010]{35J10, 65N25, 65D30, 65N30, 81Q05}

\maketitle

\begin{abstract}
  In this paper, we compute the eigenvalue problem (EVP) for the semiclassical random Schr\"odinger operators, where the random potentials are parameterized by an infinite series of random variables. After truncating the series, we introduce the multiscale finite element method (MsFEM) to approximate the resulting parametric EVP. We then use the quasi-Monte Carlo (qMC) method to calculate empirical statistics within a finite-dimensional random space. Furthermore, using a set of low-dimensional proper orthogonal decomposition (POD) basis functions, the referred degrees of freedoms for constructing multiscale basis are independent of the spatial mesh. Given the bounded assumption on the random potentials, we then derive and prove an error estimate for the proposed method. Finally, we conduct numerical experiments to validate the error estimate. In addition, we investigate the localization of eigenfunctions for the Schr\"odinger operator with spatially random potentials. The results show that our method provides a practical and efficient solution for simulating complex quantum systems governed by semiclassical random Schr\"odinger operators.
\end{abstract}

\section{Introduction}
The approximation of the eigenvalue problem (EVP) of the Schr\"odinger operator is a crucial computation task in quantum physics. When a spatially disordered potential is adopted, eigenfunctions may remain essentially localized in a small physical domain. A celebrated example is the Anderson localization~\cite{PhysRev.109.1492}, which has been extensively used to explain experimental observations, such as the metal-insulator transition of the cold atomic gas~\cite{PhysRevLett.101.255702,PhysRevLett.105.090601}, localization of optical~\cite{Riboli:11,Segev2013} and electromagnetic system~\cite{PhysRevLett.99.253902,doi:10.1126/science.1185080}.

In this paper, we consider the EVP as follows:
\begin{equation}
  \left(-\frac{\epsilon^2}{2}\Delta + V(\mathbf{x}, \boldsymbol{\omega})\right)\psi(\mathbf{x}, \boldsymbol{\omega}) = \lambda(\boldsymbol{\omega})\psi(\mathbf{x}, \boldsymbol{\omega})
  \label{equ:eigenvalue-problem}
\end{equation}
over a bounded convex domain $D \subset \mathds{R}^d$ ($d = 1,2,3$) with the periodic boundary condition, where $\epsilon$ is the semiclassical constant and $V(\mathbf{x}, \boldsymbol{\omega})$ is the random potential with $\boldsymbol{\omega} \in \Omega$ being the stochastic parameter in an infinity dimensional space $\Omega$. Here the differential operator $\Delta$ is with respect to the spatial variable $\mathbf{x}$.

We consider the stochastic parameter
\begin{equation*}
  \boldsymbol{\omega} = (\omega_j)_{j\in\mathds{N}} \in \Omega := [-\frac 12, \frac 12]^{\mathds{N}}
\end{equation*}
to be the infinite-dimensional vector of i.i.d. uniformly random variables on $[-\frac 12, \frac 12]$, and random potentials are bounded and admit the series expansion
\begin{equation}
  V(\mathbf{x}, \boldsymbol{\omega}) = v_0(\mathbf{x}) + \sum_{j=1}^{\infty} \omega_jv_j(\mathbf{x}),
  \label{equ:infinite-series-potential}
\end{equation}
where $v_j(\mathbf{x})$ ($j = 1, 2, \cdots$) are deterministic functions.

We are interested in the statistics of the eigenvalues and linear functionals of the eigenfunctions in the uncertainty quantification (UQ). More precisely, for the minimal eigenvalue $\lambda: \Omega \rightarrow \mathds{R}^+$, we aim to compute the expectation with respect to the countable product of uniform density, which is an infinite-dimensional integral defined as
\begin{equation}
  \mathds{E}_{\boldsymbol{\omega}}[\lambda] = \int_{\Omega}\lambda(\boldsymbol{\omega}) \mathrm{d}\boldsymbol{\omega} = \lim_{s\rightarrow \infty}\int_{[-\frac 12, \frac 12]^s}\lambda(\omega_1, \cdots, \omega_s, 0, \cdots)\mathrm{d}\omega_1\cdots \mathrm{d}\omega_s,
  \label{equ:integral-infinity-lambda}
\end{equation}
as well as the counterpart of the ground state $\psi$ to be
\begin{equation}
  \mathds{E}_{\boldsymbol{\omega}}[\mathcal{G}(\psi)] = \lim_{s\rightarrow \infty}\int_{[-\frac 12, \frac 12]^s}\mathcal{G}(\psi)(\cdot, \omega_1, \cdots, \omega_s, 0, \cdots)\mathrm{d}\omega_1\cdots \mathrm{d}\omega_s,
  \label{equ:integral-infinity-psi}
\end{equation}
where $\mathcal{G}$ is a linear functional in $L^2(D; \Omega)$. Numerically, the integrals \eqref{equ:integral-infinity-lambda} and \eqref{equ:integral-infinity-psi} are calculated with the setup $\omega_j = 0$ for $j > s$, which is consistent with the truncation of the potential \eqref{equ:infinite-series-potential}. Then the Monte Carlo (MC) and quasi-Monte Carlo (qMC) methods are employed to generate the random points in the high-dimensional random space. Using $N$ i.i.d. random points, MC method approximates an integral with $\mathcal{O}({N}^{-\frac 12})$ rate~\cite{liu2001monte}. Instead, using $N$ carefully chosen (deterministic) points (see example \cite{dick_kuo_sloan_2013,Tuffin+2004+617+628}), the convergence rate of qMC method can reach almost $\mathcal{O}(N^{-1})$.

To declare the challenge in computations of UQ for the random EVP \eqref{equ:eigenvalue-problem}, we denote $\boldsymbol{\omega}_s = (\omega_1, \cdots, \omega_s)$ and apply the parametric potential
\begin{equation}
  V(x, \boldsymbol{\omega}_s) = v_0(x) + \sigma\sum_{j=1}^s \frac{1}{j^q}\sin(j\pi x)\omega_j,
  \label{equ:random-potential-expansion-example}
\end{equation}
where $\sigma$ controls the strength of the randomness, and $q$ controls the decay rates of the components with different frequencies. We then need to resolve features with various frequencies in the parametric problem. For sufficiently large $s$, the degrees of freedom (dofs) required for the finite element method (FEM) would be significantly large, and this poses the computational burden on both the time and memory. Therefore, our primary task is to efficiently solve the EVP parameterized by \eqref{equ:random-potential-expansion-example}.

When the coefficients of EVP are parameterized by \eqref{equ:random-potential-expansion-example} with specifically chosen parameter values, such as the spatially disordered coefficients and multiscale coefficients,
reduced basis methods~\cite{fumagalli2016reduced,horger2017simultaneous,machiels2000output,pau2007reduced} were developed to decrease the computational complexity. Some recent progress includes the data-driven proper orthogonal decomposition (POD) methods for elliptic problems~\cite{BERTRAND2023112503,BERTRAND2023115696,https://doi.org/10.1002/nme.4533}, the localized orthogonal decomposition (LOD) and the super-LOD for the nonlinear Bose-Einstein condensate~\cite{doi:10.1137/130921520,doi:10.1137/22M1516300,doi:10.1137/1.9781611976458,PETERSEIM2024113097}, and the multiscale FEM (MsFEM) for the Schr\"odinger operator~\cite{CiCP-24-1073}. On the other hand, when the random Schr\"odinger operator is specifically considered, there is a novel approach to efficiently predict the eigenvalues and the localization of eigenstates by the localization landscape and effective potential~\cite{doi:10.1137/17M1156721,arnold2016effective,filoche2012universal}, in which only the homogeneous elliptic equation is solved. After that, a localized computation technique was also developed to resolve the localization of eigenfunctions with better accuracy for non-periodic potentials~\cite{altmann2019localized}. In further exploration of UQ problems, a combined approach, the qMC-FEM method, has been developed and thoroughly analyzed in~\cite{gilbert2019analysis}. Nevertheless, there is rarely work related to the model reduction method for the UQ problem of random EVPs, even though the model reduction methods for UQ problems of partial differential equations (PDEs) with random coefficients have made continuous progress recently, e.g., see~\cite{chen2017reduced,CHUNG2018606,doi:10.1137/18M1205844,doi:10.1137/110843253,doi:10.1137/130948136} and reference therein.

For the UQ problem of \eqref{equ:eigenvalue-problem}, our approach consists of several key steps. Firstly, random potentials are approximated by truncated series with the parameterization of stochastic parameters, and the qMC method is employed to generate the stochastic parameters. In the offline stage, we prepare the low-dimensional POD basis, which will be utilized to construct the multiscale basis corresponding to random potentials. Then in the online stage, we solve the EVP in an order-reduced system approximated by the multiscale basis. After that, the empirical statistics of eigenpairs are calculated. Although the multiscale basis is typically approximated using the standard FEM on the refined mesh, we emphasize that in our approach, the dofs in constructing the multiscale basis only rely on the dimensions of the POD basis, and are set to be $3$ in our implementations.

The approximation error of the proposed method, dubbed the MsFEM-POD method, for the EVP of random Schr\"odinger operator is a combined form that simultaneously depends on the truncated dimension $s$, the coarse mesh size $H$, the number of qMC samples $N$ and the POD error $\rho$. In particular, it exhibits the superconvergence rates with respect to $H$ in the physical space. Hence, we first prove the error bounds (\Cref{thm:convergence-rate-MsFEM}) for the multiscale solution $\lambda_{ms}$ and $\psi_{ms}$ as
\begin{equation}
    \|\psi_{ms} - \psi\|_1 \leq CH^3, \quad \|\psi_{ms} - \psi\| \leq CH^4,
    \label{equ:estimate-MsFEM-exact-eigenfunctions}
  \end{equation}
and
  \begin{equation}
    |\lambda_{ms} - \lambda| \leq CH^6,
    \label{equ:estimate-MsFEM-exact-eigenvalues}
  \end{equation}
where $\lambda$ and $\psi$ are the minimal eigenvalue and ground state, respectively. Throughout this paper, we use $(\cdot, \cdot)$ to denote the inner product in $L^2(D)$, then $\|\cdot\|$ and $\|\cdot\|_r$ ($r = 1, 2$) denote the norm in  $L^2(D)$ and $H^r(D)$ sense, respectively. In addition, we denote $H_P^1(D) = \{ v | v \in H^1(D), \text{and $v$ is periodic over $D$} \}$.

As random potentials are further considered, the corresponding multiscale basis is approximated by the POD basis. Hence, two classes of optimal problems will be repeatedly referred to hereafter in which one has been extensively used in prior studies with the dofs depending on the fine mesh, and the other one is proposed here with the referred dofs relying on the POD basis. Let $\phi_i(\mathbf{x}, \boldsymbol{\omega})$ be the reference basis function obtained by solving the original optimal problems. Then for the multiscale basis $\hat{\phi}_i(\mathbf{x}, \boldsymbol{\omega})$ approximated by the POD basis, the error bound is
\begin{equation}
    \|\phi_i(\mathbf{x}, \boldsymbol{\omega}) - \hat{\phi}_i(\mathbf{x}, \boldsymbol{\omega})\| \leq C\sqrt{\rho},
\end{equation}
where $i = 1, \cdots, N_H$, and $C$ is a constant independent of the stochastic parameter $\boldsymbol{\omega}$ and $i$. Consequently, the estimates \eqref{equ:estimate-MsFEM-exact-eigenfunctions} and \eqref{equ:estimate-MsFEM-exact-eigenvalues} are updated with an inclusion of the POD error $\rho$; for the detail see \Cref{thm:error-MsFEM-POD}.

The total error of the proposed for the UQ of EVP \eqref{equ:eigenvalue-problem} is therefore
\begin{equation}
  \sqrt{\mathds{E}_{\boldsymbol{\Delta}}\left[ |\mathds{E}_{\boldsymbol{\omega}}[\lambda] - Q_{N, s}\lambda_{s,ms}^{pod}|^2 \right]} \leq C\left( H^6 + \sqrt{\rho} + s^{-2/p+1} + N^{-\alpha} \right),
\end{equation}
where $\alpha = \min\{ 1-\delta, 1/p - 1/2 \}$ for arbitrary $\delta \in (0, 1/2)$. This result, presented in its complete form in \Cref{thm:total-error}, is given with similar results for the linear functional of the eigenfunctions.
Compared to the qMC-FEM provided in \cite{gilbert2019analysis}, by leveraging low-dimensional approximations and constructing reduced basis functions, our approach significantly reduces computational costs while maintaining high accuracy.

At the end of this paper, we conduct numerical experiments to validate the theoretical error estimate and the advantage of the efficiency of the model reduction method. Furthermore, we investigate the localization of eigenfunctions for spatially random potentials in 1D and 2D problems. An important observation is that for parameterized potentials possessing non-decaying amplitudes of high-frequency components $(q = 0)$, it strictly requires the coarse mesh size such that $H < \epsilon$. on the contrary, no such stringent constraint is needed for parameterized potentials with $q > 1$. These results showcase that our approach offers a practical and efficient solution for simulating complex quantum systems governed by semiclassical random Schr\"odinger operators.

The paper is organized as follows. We first give some useful preliminaries in Section \ref{sec:preliminaries}. Numerical algorithms are detailed in Section \ref{sec:algorithms}. The regularity of the minimal eigenvalue and ground state with respect to the stochastic parameter is analyzed in Section \ref{sec:parameteric-regularity}. The convergence analysis is given in Section \ref{sec:error-analysis}. Some experimental results are in Section \ref{sec:experiments}. Conclusions are drawn in Section \ref{sec:conclusions}.

\section{Preliminaries on the semiclassical Schr\"odinger operator with random potentials}
\label{sec:preliminaries}

Let $\hat{H}_{\omega} = -\frac{\epsilon^2}{2}\Delta + V(\mathbf{x}, \boldsymbol{\omega})$ be the random Hamiltonian operator. The solutions of \eqref{equ:eigenvalue-problem} given by $(\lambda_k, \psi_k)$ are the eigenpairs of $\hat{H}_{\omega}$, which satisfy the random weak form
\begin{equation}
  \frac{\epsilon^2}{2}\int_{D}\nabla\psi(\mathbf{x}, \boldsymbol{\omega})\nabla\phi(\mathbf{x}) \mathrm{d}\mathbf{x} + \int_D V(\mathbf{x}, \boldsymbol{\omega})\psi(\mathbf{x}, \boldsymbol{\omega})\phi(\mathbf{x}) \mathrm{d}\mathbf{x} = \lambda(\boldsymbol{\omega})\int_{D}\psi(\mathbf{x}, \boldsymbol{\omega})\phi(\mathbf{x}) \mathrm{d}\mathbf{x}.
\end{equation}
Denote the symmetric bilinear forms $\mathcal{A}(\boldsymbol{\omega}; \cdot, \cdot): H_P^1(D) \times H_P^1(D) \rightarrow \mathds{R}$ by
\begin{equation}
  \mathcal{A}(\boldsymbol{\omega}; \psi, \phi) = \frac{\epsilon^2}{2}\int_D\nabla \psi(\mathbf{x})\cdot\nabla \phi(\mathbf{x})\mathrm{d}\mathbf{x} + \int_D V(\mathbf{x}, \boldsymbol{\omega})\psi(\mathbf{x}) \phi(\mathbf{x})\mathrm{d}\mathbf{x}.
\end{equation}
Then for each $\boldsymbol{\omega} \in \Omega$, we find $\psi(\boldsymbol{\omega}) \in H_P^1(D)$ and $\lambda(\boldsymbol{\omega}) \in \mathds{R}$ such that
\begin{align}
    &\mathcal{A}(\boldsymbol{\omega}; \psi(\boldsymbol{\omega}), \phi)= \lambda(\boldsymbol{\omega})(\psi(\boldsymbol{\omega}), \phi), \qquad \forall\phi \in H_P^1(D) \label{equ:weak-form-random-EVP}
\end{align}
with a normalization constraint $\|\psi(\boldsymbol{\omega})\|= 1$.

In quantum systems, a crucial task involves identifying the minimum eigenvalue and its corresponding eigenfunction, commonly known as the ground state. We define the energy functional
\begin{equation}
  E(\phi) = \frac 12\int_{D} \frac{\epsilon^2}{2}|\nabla\phi|^2 + V(\mathbf{x}, \boldsymbol{\omega})\phi^2 \mathrm{d}\mathbf{x}.
\end{equation}
Then the ground state $\psi$ of the system is characterized as the minimizer of this energy functional, subject to the normalization constraint $\|\psi \| = 1$, i.e.,
\begin{equation}
  E(\psi) = \inf_{\|\phi\| = 1} E(\phi).
\end{equation}

We will refer to the eigenvalues of $-\Delta$ equipped with the periodic boundary condition. They are strictly positive and counting multiplicities. We denote them by
\begin{equation}
  0 < \chi_1 < \chi_2 < \cdots.
  \label{equ:eigenvalues-Laplacian}
\end{equation}
Assume random potentials are uniformly bounded with $V_{\mathrm{max}} \geq V(\mathbf{x}, \boldsymbol{\omega}) \geq V_{\mathrm{min}} \geq 0$ but $V(\mathbf{x}, \boldsymbol{\omega}) \not\equiv 0$, and we easily get the coercivity and boundedness of the bilinear form $\mathcal{A}(\boldsymbol{\omega}; \cdot, \cdot)$, which is uniform with respect to the stochastic parameter $\boldsymbol{\omega}$, i.e.,
\begin{align}
  \mathcal{A}(\boldsymbol{\omega}; v, v) &\ge c_1\|v\|_1^2, \qquad \text{for all }v \in H_P^1(D) \label{equ:coercive-A}\\
  \mathcal{A}(\boldsymbol{\omega}; u, v) &\le c_2\|u\|_1\|v\|_1, \qquad \text{for all }u, v \in H_P^1(D).\label{equ:upper-bounded-A}
\end{align}
To establish \eqref{equ:upper-bounded-A}, we use the upper bound of potentials and the Poincar\'e inequality
\begin{equation}
  \label{equ:Poincare-inequality}
  \|v\| \leq \chi_1^{-1/2}\|v\|_1, \quad \text{ for } v \in H_P^1(D).
\end{equation}
And we also have $c_2 = V_{\mathrm{max}} (1 + \epsilon^2/(2\chi_1))$.

Since the Hamiltonian operator $\hat{H}_{\omega}$ is self-adjoint and $V_{\mathrm{min}} \geq 0$, the EVP has countable-many eigenvalues $(\lambda_{k}(\boldsymbol{\omega}))_{k\in\mathds{N}}$. They are positive, have finite multiplicity, and accumulate only at infinity. We write them as
\begin{equation*}
  0 < \lambda_1(\boldsymbol{\omega}) \leq \lambda_2(\boldsymbol{\omega}) \leq \cdots
\end{equation*}
with $\lambda_k(\boldsymbol{\omega}) \rightarrow \infty$ as $k \rightarrow \infty$. For the eigenvalue $\lambda(\boldsymbol{\omega})$ we define the corresponding eigenspace
\begin{equation*}
  E(\boldsymbol{\omega}, \lambda(\boldsymbol{\omega})) := \{ \psi | \psi \text{ is an eigenfunction corresponding to }\lambda(\boldsymbol{\omega}) \}.
\end{equation*}
And for the minimal eigenvalue $\lambda_1(\boldsymbol{\omega})$, we have the following coercive-type estimate.
\begin{lemma}[\cite{gilbert2019analysis}, Lemma 3.1]
  For all $\boldsymbol{\omega} \in \Omega$ and $\lambda \in \mathds{R}$, define $\mathcal{A}_{\lambda}(\boldsymbol{\omega}; \cdot, \cdot): H_P^1(D)\times H_P^1(D) \rightarrow \mathds{R}$ to be the shifted bilinear form
  \begin{equation}
    \mathcal{A}_{\lambda}(\boldsymbol{\omega}; u, v) = \mathcal{A}(\boldsymbol{\omega}; u, v) - \lambda(u,v).
  \end{equation}
  Restricted to the $L^2$-orthogonal complement of the eigenspace corresponding to $\lambda_1(\boldsymbol{\omega})$, denoted by $E(\boldsymbol{\omega}, \lambda_1(\boldsymbol{\omega}))^{\bot}$, the $\lambda_1(\boldsymbol{\omega})$-shifted bilinear form is uniformly coercive in $\boldsymbol{\omega}$, i.e., there exists a constant $C_{gap}$ such that
  \begin{equation}
    \mathcal{A}_{\lambda_1}(\boldsymbol{\omega}; u, u) \geq C_{gap}\|u\|_1^2 \text{ for all } u \in E(\boldsymbol{\omega}, \lambda_1(\boldsymbol{\omega}))^{\bot}.
    \label{equ:coercive-A-lambda}
  \end{equation}
  \label{lem:coercive-A-lambda}
\end{lemma}

Furthermore, according to the min-max principle, the $k$th eigenvalue is to be a minimum over all the subspace $S_k \subset H_P^1(D)$:
\begin{equation}
  \lambda_k(\boldsymbol{\omega}) = \min_{S_k \subset H_P^1(D)}\max_{ 0 \neq u \in S_k } \frac{\mathcal{A}(\boldsymbol{\omega}; u, u)}{(u, u)},
\end{equation}
where $\mathrm{dim}(S_k) = k$. It can be equivalently written as
\begin{equation}
  \lambda_k(\boldsymbol{\omega}) = \min_{S_k \subset H_P^1(D)}\max_{ \substack{u \in S_k, \\ \|u\| = 1}} \mathcal{A}(\boldsymbol{\omega}; u, u),
\end{equation}
Consequently, we obtain the bound of the $k$th eigenvalue
\begin{align*}
  \lambda_{k}(\boldsymbol{\omega}) \ge c_1 \min_{S_k \subset H_P^1(D)}\max_{  \substack{u \in S_k, \\ \|u\| = 1} } \|\nabla u\|^2, \quad
  \lambda_{k}(\boldsymbol{\omega}) \le c_2 \min_{S_k \subset H_P^1(D)}\max_{  \substack{u \in S_k, \\ \|u\| = 1} } \|u\|_1^2.
\end{align*}
Using the $k$th eigenvalue of the Laplacian operator, we get the bounds of $\lambda_k(\boldsymbol{\omega})$ as
\begin{equation}
  \underline{\lambda_k} := c_1\chi_k \leq \lambda_k(\boldsymbol{\omega}) \leq c_2(\chi_k + 1) := \overline{\lambda_k}.
  \label{equ:bound-of-eigenvalue}
\end{equation}
Furthermore, since $\lambda_k(\boldsymbol{\omega}) = \mathcal{A}(\boldsymbol{\omega}; \psi_k(\boldsymbol{\omega}), \psi_k(\boldsymbol{\omega}))$, the estimate of the corresponding eigenfunction satisfies
\begin{equation}
  \|\psi_k(\boldsymbol{\omega})\|_1 \leq \sqrt{\lambda_k(\boldsymbol{\omega})/c_1} \leq \sqrt{c_2(\chi_k + 1)/c_1} := \overline{\psi_k}.
\end{equation}

\section{Numerical approximations}
\label{sec:algorithms}

\subsection{Stochastic dimension truncation}
As defined in \eqref{equ:infinite-series-potential}, the random potential $V(\mathbf{x}, \boldsymbol{\omega})$ is assumed to be an infinite series expansion.
We first truncate the infinite-dimensional problem into a $s$-dimensional problem by setting $\omega_j = 0$ for $j > s$. Denote $\boldsymbol{\omega}_s = (\omega_1, \cdots, \omega_s)$, and the random potential is truncated as
\begin{equation}
  V(\mathbf{x}, \boldsymbol{\omega}_s) = v_0 + \sum_{j=1}^s\omega_jv_j(\mathbf{x}).
\end{equation}
We then deduce a truncated symmetric bilinear form
\begin{equation*}
  \mathcal{A}_s(\boldsymbol{\omega}; u, v) = \int_D \frac{\epsilon^2}{2}\nabla u(\mathbf{x}) \cdot \nabla v(\mathbf{x}) + V(\mathbf{x}, \boldsymbol{\omega}_s) u(\mathbf{x}) v(\mathbf{x}) \mathrm{d}\mathbf{x}.
\end{equation*}
The corresponding eigenpairs $(\lambda_s(\boldsymbol{\omega}), \psi_s(\boldsymbol{\omega}))$ satisfy the parametric EVP
\begin{equation}
  \mathcal{A}_s(\boldsymbol{\omega}; \psi_s(\boldsymbol{\omega}), v) = \lambda_s(\boldsymbol{\omega})(\psi_s(\boldsymbol{\omega}), v) \qquad \text{ for all } v\in H_P^1(D)
\end{equation}
with $\|\psi_s(\boldsymbol{\omega})\| = 1$.

\subsection{MsFEM approximation}
For clarity, we consider the deterministic potential $v_0 := v_0(\mathbf{x})$ and the corresponding weak form
\begin{equation}
 a(\psi, \phi) := \frac{\epsilon^2}{2}(\nabla\psi, \nabla\phi) + (v_0\psi, \phi) = \lambda(\psi, \phi), \qquad \forall\; \phi \in H_P^1(D).
 \label{equ:weak-from}
\end{equation}

For the MsFEM, the FE basis on a coarse mesh with mesh size $H$ and the refined mesh with mesh size $h$ are required simultaneously. We consider the regular mesh $\mathcal{T}_H$ of ${D}$ and the standard $P_1$ FE space on the mesh $\mathcal{T}_H$
\begin{equation*}
  P_1(\mathcal{T}_H) = \{v\in L^2(\bar{D})| \text{ for all } K\in\mathcal{T}_H, v|_K \text{ is a polynomial of total degree} \leq 1\}.
\end{equation*}
Then the corresponding $H^1_P({D})$-confirming FE spaces are $V_h = P_1(\mathcal{T}_h)\cap H^1_P({D})$ and $V_H = P_1(\mathcal{T}_H)\cap H^1_P({D})$.

The multiscale basis functions are obtained by solving the optimal problems
\begin{align}
  \phi_i = &\argmin_{\phi\in H^1_P(D)} a(\phi, \phi),\label{equ:optimal-problem-objective}\\
  \text{s.t.}&\int_{D}\phi\phi_j^H\mathrm{d}\mathbf{x} = \alpha\delta_{ij}, \quad \forall\; 1\leq j \leq N_H,
  \label{equ:optimal-problem}
\end{align}
where $\phi_j^H \in V_H$ and $\alpha = (1, \phi_j^H)$. Here $\alpha$ is a factor to eliminate the dependence of basis functions on the mesh size, which has been elucidated by the Cl\'ement-type quasi-interpolation operator \cite{LIZHANG2024MsFEMNLSE}. Define the patches $\{D_{\ell}\}$ associated with $\mathbf{x}_i \in \mathcal{N}_H$
\begin{align*}
  &D_0(\mathbf{x}_i) := \mathrm{supp}\{\phi_i\} = \cup\{K\in\mathcal{T}_H\;|\; \mathbf{x}_i\in K\}, \\
  &D_{\ell} := \cup\{K\in\mathcal{T}_H\;|\;K\cap\overline{D_{\ell-1}}\neq \emptyset\},\quad \ell = 1,2,\cdots.
\end{align*}
The multiscale basis functions decay exponentially over the domain $D$; see the Theorem 4.2 in \cite{CiCP-24-1073}.


In this numerical framework, three fundamental assumptions on potentials are required.
\begin{assumption}
  \label{assump:assumption-for-potentials}
  \begin{enumerate}
    \item For the potential in the random Schr\"odinger operator, we assume $\|V\|_{L^{\infty}(D; \Omega)} = V_{max} < \infty$ and $H\sqrt{V_{max}}/\epsilon \lesssim 1$.
    \item For some $0 < p < 1$, it holds $\sum_{j=1}^{\infty}\|v_j\|_{L^{\infty}}^p < \infty$.
    \item $v_j \in W^{1, \infty}(D)$ for $j \ge 0$ and $\sum_{j=1}^{\infty} \|v_j\|_{W^{1, \infty}(D)} < \infty$.
  \end{enumerate}
\end{assumption}

The first assumption gives a necessary condition to the optimal problems \eqref{equ:optimal-problem-objective}-\eqref{equ:optimal-problem}. And the others ensure that the parameterized EVP is well-posed.

On the refined mesh, the multiscale basis functions are expressed as
\begin{equation}
  \phi_i = \sum_{k=1}^{N_h}c_k^i\phi_k^h,
  \label{equ:multiscal-basis-function-fine-mesh}
\end{equation}
where $i$ traverses all the coarse grid nodes. The eigenfunction is therefore approximated by $\psi_{ms} = \sum_{i=1}^{N_H}u_i\phi_i$ in the space $V_{ms} = span\{\phi_1, \cdots, \phi_{N_H}\}$, and the corresponding discretized equations are
\begin{equation*}
  \frac{\epsilon^2}{2}\sum_{i=1}^{N_H}(\nabla\phi_i, \nabla\phi_j)u_i + \sum_{i=1}^{N_H}(v_0\phi_i, \phi_j)u_i = \lambda \sum_{i=1}^{N_H}(\phi_i, \phi_j)
\end{equation*}
with $j = 1, \cdots, N_H$.
Denote the matrices $M^h = [M_{ij}^h]$ with $M_{ij}^h = (\phi_i^h, \phi_j^h)$, $S^h = [S_{ij}^h]$ with $S_{ij}^h = (\nabla\phi_i^h, \nabla\phi_j^h)$, $V_{ij}^h = (v_0\phi_i^h, \phi_j^h)$, $A = [A_{ij}]$ with $A_{ij} = (\phi_i^H, \phi_j^h)$, $C = [C_{ij}]$ with $C_{ij} = c_i^j$. The coefficients in multiscale basis function \eqref{equ:multiscal-basis-function-fine-mesh} are solved from the equality-constrained quadratic programming
\begin{equation}
  \left\{\begin{aligned}
    \min\; &C^T G C \\
    s.t. \; &AC = \alpha I,
  \end{aligned}\right.
\end{equation}
where $G = \frac{\epsilon^2}{2}S^h + V^h$ and $I$ is the unit matrix with size of $N_H\times N_H$.

With the random potential further considered, the direct combination of the MsFEM and qMC method is outlined as the following algorithm.
\begin{algorithm}[ht]
  \caption{The qMC-MsFEM for the EVP of the random Schr\"odinger operator.}
  \label{alg:MsFEM-deterministic-EVP}
  \begin{algorithmic}[1]
    \REQUIRE Stochastic samples $\{{\omega}^j\}_{j=1}^N$, coarse mesh $\mathcal{T}_H$, refined mesh $\mathcal{T}_h$
    \ENSURE Expectation of eigenpairs $(\mathds{E}(\lambda_{ms}), \mathds{E}(\psi_{ms}))$
    \FOR{each $j \in [1, N]$}
    \STATE Solve optimal problems \eqref{equ:optimal-problem-objective}-\eqref{equ:optimal-problem} and construct multiscale basis $\{\phi_i({\omega}^j)\}_{i=1}^{N_H}$;
    \STATE Find $\lambda_{ms}({\omega}^j) \in \mathds{R}^{+}$ and $\psi_{ms}({\omega}^j) \in V_{ms} := span\{\phi_i({\omega}^j)\}_{i=1}^{N_H}$ such that
    \begin{equation}
      \mathcal{A}({\omega}^j; \psi_{ms}({\omega}^j), \phi_{ms}) = \lambda_{ms}({\omega}^j) (\psi_{ms}({\omega}^j), \phi_{ms}), \quad\forall \phi_{ms}\in V_{ms}.
      \label{equ:numerical-solution-msfem}
    \end{equation}
    \ENDFOR
    \STATE Compute the expectation $(\mathds{E}(\lambda_{ms}), \mathds{E}(\psi_{ms}))$;
  \end{algorithmic}
\end{algorithm}

\subsection{A POD reduction method}
In \Cref{alg:MsFEM-deterministic-EVP}, the construction of the multiscale basis is repeated for all realizations of the random potential. In the worst case, the dofs of each optimal problem are $N_h$. This takes the computational burden for computations. Here we propose a POD reduction method to construct the multiscale basis, where the dofs involved are independent of the spatial partitions.

Before the formal algorithm is given, we briefly review the POD method. Let $X$ be a Hilbert space equipped with the inner product $(\cdot, \cdot)_{X}$ and norm $\|\cdot\|_X$. For $u_1, \cdots, u_n \in X$ we refer to $\mathcal{V} = span\{u_1, \cdots, u_n\}$ as ensemble consisting of the snapshots $\{u_j\}_{j=1}^n$.  Let $\{\varphi_k\}_{k=1}^{m}$ be an orthonormal basis of $\mathcal{V}$ with $m = \dim\mathcal{V}$. Then the snapshots can be expressed into $u_j = \sum_{k=1}^m (u_j, \varphi_k)_X\varphi_k$
for $j = 1, \cdots, n$.
The method consists of choosing the orthonormal basis such that for every $\ell \in \{1, \cdots, m\}$ the following mean square error is minimized
\begin{equation}
  \begin{aligned}
    &\min_{\{\varphi_k\}_{k=1}^{\ell}} \frac 1n\sum_{j=1}^n \left\| u_j - \sum_{k=1}^{\ell}(u_j, \varphi_k)_X \varphi_k \right\|_X^2 \\
    &\quad \mathrm{s. t. } \quad  (\varphi_i, \varphi_j)_X = \delta_{ij}, \text{ for } 1 \leq i, j \leq \ell.
  \end{aligned}
\end{equation}
A solution $\{\varphi_k\}_{k=1}^{\ell}$ is called a POD basis of rank $\ell$.

Define the correlation matrix $K = [K_{ij}]$ with $K_{ij} = \frac 1n(u_i, u_j)_X$.
It is positive semi-definite and has rank $m$. Let $\sigma_1 \ge \cdots \ge \sigma_m > 0$ denote positive eigenvalues of $K$ and $v_1, \cdots, v_m$ denote eigenvectors. Then a POD basis is given by
\begin{equation*}
  \varphi_k = \frac{1}{\sqrt{\sigma_k}}\sum_{j=1}^n (v_k)_j u_j,
\end{equation*}
where $(v_k)_j$ is the $j$-th component of the eigenvector $v_k$. 
Consequently, the POD approximation is described by the proposition as follows.
\begin{proposition}[\cite{benner2015survey,Holmes_Lumley_Berkooz_1996}]
  \label{prop:POD-approximation}
  For all of the snapshots, the approximation error of the POD basis with dimension $m$ satisfies
  \begin{equation}
    \frac{\sum_{j=1}^n \left\| u_j - \sum_{k=1}^{\ell}(u_j, \varphi_k)_X\varphi_k \right\|_X^2}{\sum_{j=1}^n \left\| u_j \right\|_X^2} = \frac{\sum_{k=\ell + 1}^m\sigma_k}{\sum_{k=1}^m\sigma_k}.
  \end{equation}
\end{proposition}
In the POD method, $\ell (\ll n)$ is typically determined such that
\begin{equation}
  \frac{\sum_{k=\ell + 1}^m\sigma_k}{\sum_{1}^m\sigma_k} < \rho,
\end{equation}
where $\rho$ is a user-specified tolerance, often taken to be 0.1\% or less.

When the POD method is employed, we let $\{ V(\mathbf{x}, \omega^j) \}_{j = 1}^Q$ be the parameterized potentials with $Q$ the number of samples. Solve the optimal problem \eqref{equ:optimal-problem-objective}-\eqref{equ:optimal-problem} and we obtain random multiscale functions $\phi_i(\mathbf{x}, \omega^j)$, where $i = 1, \cdots, N_H$ and $j = 1, \cdots, Q$. At $\mathbf{x}_i$, $\zeta_i^0 = \frac 1{Q}\sum_{j=1}^Q \phi_i(\mathbf{x}, \omega^j)$ is the mean of random multiscale functions, and $\tilde{\phi}_i(\mathbf{x}, \omega^j) = \phi_i(\mathbf{x}, \omega^j) - \zeta_i^0$ are fluctuations. For each $i$, employ the POD method to $\{\tilde{\phi}_i(\mathbf{x}, \omega^j)\}_{j=1}^Q$ and build the order-reduced set $\{\zeta_i^1(\mathbf{x}), \cdots, \zeta_i^{m_i}(\mathbf{x})\}$ with $m_i \ll Q$.

Then, for a stochastic sample $\boldsymbol{\omega}$, the multiscale basis can be approximated as
\begin{equation}
  \phi_i(\mathbf{x}, \boldsymbol{\omega}) \approx \sum_{l = 0}^{m_i} c_i^l(\boldsymbol{\omega}) \zeta_i^l(\mathbf{x}),
  \label{equ:basis-functions-random-operator}
\end{equation}
where $c_i^l(\boldsymbol{\omega})$ are to be determined with $i = 1, \cdots, N_H$ and $l = 0, \cdots, m_i$. The eigenfunction is further approximated by
\begin{equation}
  \psi^{\epsilon}(\mathbf{x}, \boldsymbol{\omega}) \approx \sum_{i=1}^{N_H}u_i\sum_{l = 0}^{m_i} c_i^l(\boldsymbol{\omega}) \zeta_i^l(\mathbf{x}),
\end{equation}
in which $u_i$ and $c_i^l$ are unknown. Next, we determine the unknowns $c_i^l$, leaving the discrete EVP with dofs $N_H$ to be solved.

Notice that the POD basis can be expressed into
\begin{equation}
  \zeta_i^l(\mathbf{x}) = \sum_{j=1}^{N_h} c_{i,j}^l \phi_j^h,
\end{equation}
and we can easily get $a(\zeta_i^l, \zeta_j^l) = \frac{\epsilon^2}{2}(\nabla\zeta_i^l, \nabla\zeta_j^l) + (v_0\zeta_i^l, \zeta_j^l)$. Meanwhile, owing to $\zeta_i^0 = \frac 1{Q}\sum_{j=1}^Q \phi_i(\mathbf{x}, \omega^j)$, it holds $(\zeta_i^0, \phi_j^H) = \alpha\delta_{i,j}$. And for $k \neq 0$, since
\begin{equation*}
  \zeta_i^k = \frac{1}{\sqrt{\sigma_k}}\sum_{j=1}^{Q} (v_k)_j \tilde{\phi}_i(\mathbf{x}, \omega^j) = \frac{1}{\sqrt{\sigma_k}}\sum_{j=1}^{Q} (v_k)_j \left( {\phi}_i(\mathbf{x}, \omega^j) - \frac 1{Q}\sum_{l=1}^Q \phi_i(\mathbf{x}, \omega^l) \right),
\end{equation*}
there holds $(\zeta_i^k, \phi_{j}^H) = 0$ for all $i, j = 1, \cdots, N_H$.
We therefore get the reduced optimal problem that the dofs depend on the dimension of the POD basis:
\begin{equation}
  \begin{aligned}
    &\min a\left( \sum_{l = 0}^{m_i} c_i^l \zeta_i^l(\mathbf{x}), \sum_{l = 0}^{m_i} c_i^l \zeta_i^l(\mathbf{x}) \right),\\
    &\text{s.t.}\qquad\int_{D} \sum_{l = 0}^{m_i} c_i^l \zeta_i^l(\mathbf{x}) \phi_i^H\mathrm{d}\mathbf{x} = \alpha.
  \end{aligned}
  \label{equ:optimal-problem2}
\end{equation}
We emphasized that here only one constraint is effective. The formal algorithm is then outlined as follows.
\begin{algorithm}[ht]
  \caption{The qMC with MsFEM-POD method for the EVP of the random Schr\"odinger operator.}
  \label{alg:MsFEM-POD-random-EVP}
  \begin{algorithmic}[1]
    \REQUIRE Random samples $\{\omega^j\}_{j=1}^N$, $Q$, coarse mesh $\mathcal{T}_H$, fine mesh $\mathcal{T}_h$, $i = 1, \cdots, N_H$
    \ENSURE Expectation of eigenpairs $(\mathds{E}(\lambda_{ms}), \mathds{E}(\psi_{ms}))$
    \FOR{each $j \in [1, Q]$}
    \STATE Solve optimal problems \eqref{equ:optimal-problem-objective}-\eqref{equ:optimal-problem} and generate basis sets $\{\phi_i^j\}_{j=1}^Q$ for all $i$;
    \ENDFOR
    \STATE Employ POD method to construct the order-reduced set $\Xi_i = \{\zeta_i^0(\mathbf{x}), \zeta_i^1(\mathbf{x}), \cdots, \zeta_i^{m_i}(\mathbf{x})\}$;
    \STATE Construct the new optimal problems \eqref{equ:optimal-problem2} by $\Xi_i$;
    \FOR{each $j \in [1, N]$}
    \STATE For the potential parameterized by $\omega^j$, solve the optimal problem \eqref{equ:optimal-problem2} and generate the multiscale basis $\{\hat{\phi}_i(\omega^j)\}_{i=1}^{N_H}$;
    \STATE Find $\lambda_{ms}({\omega}^j) \in \mathds{R}^{+}$ and $\psi_{ms}({\omega}^j) \in V_{ms}^{pod} := span\{\hat{\phi}_i({\omega}^j)\}_{i=1}^{N_H}$ such that
    \begin{equation}
      \mathcal{A}({\omega}^j; \psi_{ms}({\omega}^j), \phi_{ms}) = \lambda_{ms}({\omega}^j) (\psi_{ms}({\omega}^j), \phi_{ms}), \quad\forall \phi_{ms}\in V_{ms}^{pod}.
    \end{equation}
    \ENDFOR
    \STATE Compute the expectation $(\mathds{E}(\lambda_{ms}), \mathds{E}(\psi_{ms}))$.
  \end{algorithmic}
\end{algorithm}

\subsection{The qMC method}
The qMC method is a popular approach for approximating high-dimensional integrals, primarily due to its better convergence rate than the conventional MC method.
We consider the $s$-dimensional integral ($s$ usually very large)
\begin{equation}
  I_s(f) = \int_{[-\frac 12, \frac 12]^s} f(\boldsymbol{\omega}) \mathrm{d}\boldsymbol{\omega}.
  \label{equ:s-dimensional-integral}
\end{equation}
This integral cannot be analytically calculated and we use a class of qMC rules called randomly shifted rank-1 lattice rules to calculate it numerically. The integral points are constructed using a generating vector $z \in \mathds{N}^s$ and a uniformly distributed random shift $\boldsymbol{\Delta} \in [0, 1]^s$. We therefore obtain the approximation of \eqref{equ:s-dimensional-integral}
\begin{equation*}
  Q_{N, s}(f) = \frac 1N\sum_{j=1}^{N-1}f\left(\Big\{ \frac{jz}{N} + \boldsymbol{\Delta} \Big\} - \frac 12\right),
\end{equation*}
in which the braces indicate that the fractional part of each component is taken.

The error estimate of randomly shifted lattice rules requires the integrand to belong to a weighted Sobolev space. Denote $\mathcal{W}_{s, \boldsymbol{\gamma}}$ the $s$-dimensional weighted Sobolev space in which functions are square-integrable mixed first derivatives, and the concerned norm depends on a set of positive real weights. Let $\boldsymbol{\gamma} = \{\gamma_{\mathfrak{u}} > 0: \mathfrak{u} \subset\{1,2,\cdots, s\}\}$ be a collection of weights, and $W_{s, \boldsymbol{\gamma}}$ be the $s$-dimensional "unanchored" weighted Sobolev space equipped the norm
\begin{equation}
  \|f\|_{s, \boldsymbol{\gamma}} = \sum_{\mathfrak{u}\subset \{1:s\}} \frac 1{\gamma_{\mathfrak{u}}} \int_{[-\frac 12, \frac 12]^{|\mathfrak{u}|}} \left( \int_{[-\frac 12, \frac 12]^{s-|\mathfrak{u}|}} \frac{\partial^{|\mathfrak{u}|}}{\partial\omega_{\mathfrak{u}}} f(\omega) \mathrm{d}\omega_{-\mathfrak{u}} \right)^2 \mathrm{d}\omega_{\mathfrak{u}},
  \label{equ:qMC-general-estimate}
\end{equation}
where $\{ 1:s \} = \{1,2,\cdots, s\}$, $\omega_{\mathfrak{u}} = (\omega_j)_{j\in\mathfrak{u}}$ and $\omega_{-\mathfrak{u}} = (\omega_j)_{j\in\{1:s\}\backslash\mathfrak{u}}$. The root-mean-square error of such qMC approximation is
\begin{equation}
  \sqrt{ \mathds{E}_{\boldsymbol{\Delta}}\left( |I_s(f) - Q_{N,s}(f)| \right) } \leq \left( \frac 1{\varphi(N)} \sum_{\emptyset \neq \mathfrak{u}\subset \{1:s\}}\gamma_{\mathfrak{u}}^{\eta} \left( \frac{2\zeta(2\eta)}{(2\pi^2)^{\eta}} \right)^{|\mathfrak{u}|} \right)^{\frac 1{2\eta}} \|f\|_{s,\boldsymbol{\gamma}}
\end{equation}
for all $\eta \in (\frac 12, 1]$, where the expectation $\mathds{E}_{\boldsymbol{\Delta}}$ is taken with respect to the random shift $\boldsymbol{\Delta}$, $\varphi(N)$ is the Euler totient function with $\varphi(N) = |\{1 \leq \xi \leq N: \mathrm{gcd}(\xi, N) = 1\}|$, and $\zeta(x) = \sum_{k=1}^{\infty}k^{-x}$ for $x > 1$ is the Riemann zeta function. Note that $\varphi(N) = N - 1$ for prime $N$.

\section{Parametric regularity}
\label{sec:parameteric-regularity}
As indicated by \eqref{equ:qMC-general-estimate}, the norms $\|\lambda\|_{s, \boldsymbol{\gamma}}$ and $\|\psi\|_{s, \boldsymbol{\gamma}}$ are required for the error analysis of the qMC approximation. Let $\mathbf{m} = (m_j)_{j\in \mathds{N}}$, $\boldsymbol{\nu} = (\nu_j)_{j\in \mathds{N}}$. We define the multi-index notations: $\boldsymbol{\nu}! = \prod_{j\ge 1}\nu_j!$; $\boldsymbol{\nu} - \mathbf{m} = (\nu_j - m_j)_{j\in\mathds{N}}$; $\mathbf{m} \leq \boldsymbol{\nu}$ if $m_j \leq \nu_j$ for all $j \in \mathds{N}$. The following lemma gives the bound on the derivative of minimal eigenvalue $\lambda$ with respect to the stochastic variable $\boldsymbol{\omega}$, as well as the $L^2$-bound and $H^1$-bound on the derivative of the ground state.

\begin{lemma}
  Let $\boldsymbol{\nu}$ be a multi-index satisfying $|\boldsymbol{\nu}| \ge 0$. Then, for all $\boldsymbol{\omega} \in \Omega$, the derivative of the minimal eigenvalue with respect to $\boldsymbol{\omega}$ is bounded by
  \begin{equation}
    |\partial_{\boldsymbol{\omega}}^{\boldsymbol{\nu}}\lambda| \leq \frac{C_1(|\boldsymbol{\nu}|!)^{1+\epsilon}}{C_{gap}^{|\boldsymbol{\nu}|-1}} \prod_{j}(\|v_j\|_{\infty})^{\nu_j},
    \label{equ:bound-eigenvalue}
  \end{equation}
  and the derivative of the ground state satisfies
  \begin{align}
    \|\partial_{\boldsymbol{\omega}}^{\boldsymbol{\nu}}\psi\| \leq \frac{C_2 (|\boldsymbol{\nu}|!)^{1+\epsilon}}{C_{gap}^{|\boldsymbol{\nu}|}} \prod_{j}(\|v_j\|_{\infty})^{\nu_j}, \quad \|\partial_{\boldsymbol{\omega}}^{\boldsymbol{\nu}}\psi\|_1 \leq \frac{C_3\overline{\psi}(|\boldsymbol{\nu}|!)^{1+\epsilon}}{(C_{gap}\chi_1)^{|\boldsymbol{\nu}|}} \prod_{j}(\|v_j\|_{\infty})^{\nu_j},
    \label{equ:bound-eigenfunction}
  \end{align}
  where $\epsilon \in (0, 1)$, $C_1$, $C_2$ and $C_3$ are finite constants.
  \label{lem:parametric-regularity}
\end{lemma}

\begin{proof}
  According to the definitions of $\overline{\lambda}$ and $\overline{\psi}$, the bounds \eqref{equ:bound-eigenvalue} and \eqref{equ:bound-eigenfunction} clearly hold for $\boldsymbol{\nu} = \mathbf{0}$.
  For $\boldsymbol{\nu} \ne \mathbf{0}$, taking the derivatives for the \eqref{equ:eigenvalue-problem} with respect to $\boldsymbol{\omega}$, and employing the Leibniz general product rule yields
  \begin{equation}
    -\frac{\epsilon^2}{2} \Delta \partial_{\boldsymbol{\omega}}^{\boldsymbol{\nu}}\psi + V(\mathbf{x}, \boldsymbol{\omega})\partial_{\boldsymbol{\omega}}^{\boldsymbol{\nu}}\psi + \sum_{j=1}^{\infty}v_j\partial_{\boldsymbol{\omega}}^{\boldsymbol{\nu} - \mathbf{e}_j}\psi = \sum_{\mathbf{m}\leq \boldsymbol{\nu}} \binom{\boldsymbol{\nu}}{\mathbf{m}}\partial_{\boldsymbol{\omega}}^{\mathbf{m}}\lambda \partial_{\boldsymbol{\omega}}^{\boldsymbol{\nu} - \mathbf{m}}\psi.
    \label{equ:derivative-EVP-wrt-omega}
  \end{equation}
  Separating out the $\partial_{\boldsymbol{\omega}}^{\boldsymbol{\nu}}\lambda$ and using $\|\psi\| = 1$ yields
  \begin{align*}
    \partial_{\boldsymbol{\omega}}^{\boldsymbol{\nu}}\lambda = &\frac{\epsilon^2}{2} (\nabla \partial_{\boldsymbol{\omega}}^{\boldsymbol{\nu}}\psi, \nabla\psi) + (V(\mathbf{x}, \boldsymbol{\omega})\partial_{\boldsymbol{\omega}}^{\boldsymbol{\nu}}\psi, \psi)
    \\
    &+ \sum_{j=1}^{\infty} (v_j\partial_{\boldsymbol{\omega}}^{\boldsymbol{\nu} - \mathbf{e}_j}\psi, \psi) - \sum_{ \mathbf{m} < \boldsymbol{\nu}} \binom{\boldsymbol{\nu}}{\mathbf{m}} \partial_{\boldsymbol{\omega}}^{\mathbf{m}}\lambda (\partial_{\boldsymbol{\omega}}^{\boldsymbol{\nu} - \mathbf{m}}\psi, \psi).
  \end{align*}
  Due to $\frac{\epsilon^2}{2} (\nabla \partial_{\boldsymbol{\omega}}^{\boldsymbol{\nu}}\psi, \nabla\psi) + (V(\mathbf{x}, \boldsymbol{\omega})\partial_{\boldsymbol{\omega}}^{\boldsymbol{\nu}}\psi, \psi) - \lambda(\partial_{\boldsymbol{\omega}}^{\boldsymbol{\nu}}\psi, \psi) = 0$,
  we have
  \begin{align}
    |\partial_{\boldsymbol{\omega}}^{\boldsymbol{\nu}}\lambda| &\leq \sum_{j=1}^{\infty} \nu_j (v_j\partial_{\boldsymbol{\omega}}^{\boldsymbol{\nu} - \mathbf{e}_j}\psi, \psi) - \sum_{\mathbf{0} \neq \mathbf{m} < \boldsymbol{\nu}} {\boldsymbol{\nu} \choose \mathbf{m}} \partial_{\boldsymbol{\omega}}^{\mathbf{m}}\lambda (\partial_{\boldsymbol{\omega}}^{\boldsymbol{\nu} - \mathbf{m}}\psi, \psi) \nonumber\\
    &\leq \sum_{j=1}^{\infty} \nu_j \|v_j\|_{\infty} \|\partial_{\boldsymbol{\omega}}^{\boldsymbol{\nu} - \mathbf{e}_j}\psi\| + \sum_{\mathbf{0} \neq \mathbf{m} < \boldsymbol{\nu}} {\boldsymbol{\nu} \choose \mathbf{m}} |\partial_{\boldsymbol{\omega}}^{\mathbf{m}}\lambda| \| \partial_{\boldsymbol{\omega}}^{\boldsymbol{\nu} - \mathbf{m}}\psi\| \label{equ:L2-bound-partial-nu-lambda}\\
    &\leq \chi_1^{-1/2}\sum_{j=1}^{\infty} \nu_j \|v_j\|_{\infty} \|\partial_{\boldsymbol{\omega}}^{\boldsymbol{\nu} - \mathbf{e}_j}\psi\|_1 + \chi_1^{-1/2}\sum_{\mathbf{0} \neq \mathbf{m} < \boldsymbol{\nu}} {\boldsymbol{\nu} \choose \mathbf{m}} |\partial_{\boldsymbol{\omega}}^{\mathbf{m}}\lambda| \| \partial_{\boldsymbol{\omega}}^{\boldsymbol{\nu} - \mathbf{m}}\psi\|_1.
    \label{equ:H1-bound-partial-nu-lambda}
  \end{align}
  This indicates that the bound $|\partial_{\boldsymbol{\omega}}^{\boldsymbol{\nu}}\lambda|$ depends on the lower order derivatives of both the minimal eigenvalue $\lambda$ and the ground state $\psi$.

  Next, we compute the bound of $\|\partial_{\boldsymbol{\omega}}^{\boldsymbol{\nu}}\psi\|_1$. Since $\|\psi\| = 1$, we have
  \begin{equation*}
    0 = \partial_{\boldsymbol{\omega}}^{\boldsymbol{\nu}} (\psi, \psi) = \sum_{\mathbf{m}\leq \boldsymbol{\nu}} {\boldsymbol{\nu} \choose \mathbf{m}} (\partial_{\boldsymbol{\omega}}^{\mathbf{m}}\psi, \partial_{\boldsymbol{\omega}}^{\boldsymbol{\nu} - \mathbf{m}}\psi),
  \end{equation*}
  which infers that
  \begin{align}
    |(\partial_{\boldsymbol{\omega}}^{\boldsymbol{\nu}}\psi, \psi)| &= \Big| -\frac 12 \sum_{\mathbf{0} \ne \mathbf{m} < \boldsymbol{\nu}} {\boldsymbol{\nu} \choose \mathbf{m}} (\partial_{\boldsymbol{\omega}}^{\mathbf{m}}\psi, \partial_{\boldsymbol{\omega}}^{\boldsymbol{\nu} - \mathbf{m}}\psi) \Big| \nonumber \\
    &\leq \frac{1}{2}\sum_{\mathbf{0} \ne \mathbf{m} < \boldsymbol{\nu}} {\boldsymbol{\nu} \choose \mathbf{m}} \|\partial_{\boldsymbol{\omega}}^{\mathbf{m}}\psi\| \|\partial_{\boldsymbol{\omega}}^{\boldsymbol{\nu} - \mathbf{m}}\psi\|
    \leq \frac{1}{2\chi_1}\sum_{\mathbf{0} \ne \mathbf{m} < \boldsymbol{\nu}} {\boldsymbol{\nu} \choose \mathbf{m}} \|\partial_{\boldsymbol{\omega}}^{\mathbf{m}}\psi\|_1 \|\partial_{\boldsymbol{\omega}}^{\boldsymbol{\nu} - \mathbf{m}}\psi\|_1.
    \label{equ:bound-inner-product-nu-derivative}
  \end{align}
  In the eigenspace, we have the decomposition
  \begin{equation}
    \partial_{\boldsymbol{\omega}}^{\boldsymbol{\nu}}\psi = \sum_{k\in\mathds{N}}(\partial_{\boldsymbol{\omega}}^{\boldsymbol{\nu}}\psi_k, \psi_k)\psi_k = (\partial_{\boldsymbol{\omega}}^{\boldsymbol{\nu}}\psi, \psi)\psi + \tilde{\psi},
    \label{equ:eigenstate-decomposition}
  \end{equation}
  so that $\tilde{\psi} \in E(\boldsymbol{\omega}, \lambda(\boldsymbol{\omega}))^{\bot}$, which infers that
  \begin{align}
    \| \partial_{\boldsymbol{\omega}}^{\boldsymbol{\nu}}\psi \| \leq |(\partial_{\boldsymbol{\omega}}^{\boldsymbol{\nu}}\psi, \psi)| + \|\tilde{\psi}\|,
    \label{equ:bound-psi-L2}
  \end{align}
  as well as
  \begin{align}
    \| \partial_{\boldsymbol{\omega}}^{\boldsymbol{\nu}}\psi \|_1 \leq |(\partial_{\boldsymbol{\omega}}^{\boldsymbol{\nu}}\psi, \psi)| \overline{\psi} + \|\tilde{\psi}\|_1.
    \label{equ:bound-psi-H1}
  \end{align}
  We first compute the $L^2$-bound. Owing to
  \begin{align*}
    \mathcal{A}(\partial_{\boldsymbol{\omega}}^{\boldsymbol{\nu}}\psi, \tilde{\psi}) - \lambda(\partial_{\boldsymbol{\omega}}^{\boldsymbol{\nu}}\psi, \tilde{\psi}) &= (\partial_{\boldsymbol{\omega}}^{\boldsymbol{\nu}}\psi, \tilde{\psi})(\mathcal{A}(\psi, \tilde{\psi}) - \lambda(\psi, \tilde{\psi})) + \mathcal{A}(\tilde{\psi}, \tilde{\psi}) - \lambda(\tilde{\psi}, \tilde{\psi}) \\
    &= \mathcal{A}(\tilde{\psi}, \tilde{\psi}) - \lambda(\tilde{\psi}, \tilde{\psi}) \ge C_{gap}\|\tilde{\psi}\|_1^2 \ge C_{gap}\|\tilde{\psi}\|^2.
  \end{align*}
  Taking inner product of \eqref{equ:derivative-EVP-wrt-omega} with $\tilde{\psi}$ yields
  \begin{equation*}
    \mathcal{A}(\partial_{\boldsymbol{\omega}}^{\boldsymbol{\nu}}\psi, \tilde{\psi}) - \lambda(\partial_{\boldsymbol{\omega}}^{\boldsymbol{\nu}}\psi, \tilde{\psi}) = \sum_{\mathbf{0}\ne \mathbf{m} < \boldsymbol{\nu}}{\boldsymbol{\nu} \choose \mathbf{m}}\partial_{\boldsymbol{\omega}}^{\mathbf{m}}\lambda (\partial_{\boldsymbol{\omega}}^{\boldsymbol{\nu}-\mathbf{m}}\psi, \tilde{\psi}) - \sum_{j=1}^{\infty} \nu_j(v_j\partial_{\boldsymbol{\omega}}^{\boldsymbol{\nu} - \mathbf{e}_j}\psi, \tilde{\psi}),
  \end{equation*}
  in which we have used the fact that $(\psi, \tilde{\psi}) = 0$.
  We then arrive at
  \begin{equation}
    \|\tilde{\psi}\| \leq \frac{1}{C_{gap}} \left( \sum_{\mathbf{0}\ne \mathbf{m} < \boldsymbol{\nu}}{\boldsymbol{\nu} \choose \mathbf{m}} |\partial_{\boldsymbol{\omega}}^{\mathbf{m}}\lambda| \|\partial_{\boldsymbol{\omega}}^{\boldsymbol{\nu}-\mathbf{m}}\psi\| + \sum_{j=1}^{\infty} \nu_j \|v_j\|_{\infty} \|\partial_{\boldsymbol{\omega}}^{\boldsymbol{\nu} - \mathbf{e}_j}\psi\| \right).
    \label{equ:bound-tilde-psi}
  \end{equation}

  Substituting the two bounds \eqref{equ:bound-inner-product-nu-derivative} and \eqref{equ:bound-tilde-psi} into \eqref{equ:bound-psi-L2}, we derive the bound on the derivative of the ground state
  \begin{equation}
    \|\partial_{\boldsymbol{\omega}}^{\boldsymbol{\nu}}\psi\| \leq \sum_{\mathbf{0} \ne \mathbf{m} < \boldsymbol{\nu}} {\boldsymbol{\nu}
    \choose \mathbf{m}} \left( \frac{1}{2} \|\partial_{\boldsymbol{\omega}}^{\mathbf{m}}\psi\| + \frac{|\partial_{\boldsymbol{\omega}}^{\mathbf{m}}\lambda|}{C_{gap}} \right) \|\partial_{\boldsymbol{\omega}}^{\boldsymbol{\nu} - \mathbf{m}}\psi\| +
    \frac{1}{C_{gap}} \sum_{j=1}^{\infty} \nu_j \|v_j\|_{\infty} \|\partial_{\boldsymbol{\omega}}^{\boldsymbol{\nu} - \mathbf{e}_j}\psi\|.
    \label{equ:bound-partial-nu-psi-L2}
  \end{equation}

  Furthermore, applying the Poincar\'e inequality and repeating the above procedures yields the $H^1$-bound
  \begin{align*}
    \|\partial_{\boldsymbol{\omega}}^{\boldsymbol{\nu}}\psi\|_1 \leq &\sum_{\mathbf{0} \ne \mathbf{m} < \boldsymbol{\nu}} {\boldsymbol{\nu}
    \choose \mathbf{m}} \left( \frac{\overline{\psi}
    }{2\chi_1} \|\partial_{\boldsymbol{\omega}}^{\mathbf{m}}\psi\|_1 + \frac{|\partial_{\boldsymbol{\omega}}^{\mathbf{m}}\lambda|}{C_{gap}\chi_1} \right) \|\partial_{\boldsymbol{\omega}}^{\boldsymbol{\nu} - \mathbf{m}}\psi\|_1 \\
    &+
    \frac{1}{C_{gap}\chi_1} \sum_{j=1}^{\infty} \nu_j \|v_j\|_{\infty} \|\partial_{\boldsymbol{\omega}}^{\boldsymbol{\nu} - \mathbf{e}_j}\psi\|_1.
    \label{equ:bound-partial-nu-psi-H1}
  \end{align*}
  Next, with similar induction steps as Lemma 3.4 in \cite{gilbert2019analysis}, an application of the induction argument yields \eqref{equ:bound-eigenvalue} and \eqref{equ:bound-eigenfunction}.
\end{proof}

\section{Convergence analysis}
\label{sec:error-analysis}
The error analysis of this part mainly concerns the approximation error of \Cref{alg:MsFEM-POD-random-EVP}. Remove the POD error from the results and produce the error estimate for \Cref{alg:MsFEM-deterministic-EVP}.
To begin with, we derive a priory estimate for the variational approximation of the deterministic Schr\"odinger operator.

\subsection{Convergence analysis of the MsFEM for the EVP}
\label{subsec:MsFEM-deternimistic-operator}

Denote $V_H$ a family of finite-dimensional subspace of $H_P^1(D)$ such that
\begin{equation}
  \min\{ \|\psi - \psi_H\|_1; \psi_H \in V_H \} \xrightarrow[H\rightarrow 0^+]{} 0
\end{equation}
and define the variational approximation of the deterministic Schr\"odinger operator
\begin{equation}
  \inf \{ E(\psi_H); \psi_H \in V_H, \|\psi_H\| = 1 \}.
\end{equation}
This problem has at least one minimizer $\psi_H$ such that for some $\lambda \in \mathds{R}$
\begin{equation}
  \langle \hat{H}\psi_H - \lambda\psi_H, v \rangle = 0, \quad \forall v \in V_H,
\end{equation}
where $\hat{H} = \frac{\epsilon^2}{2}\Delta + v_0(\mathbf{x})$.
Assume $v_0(\mathbf{x}) \in L^{\infty}(D)$, for all $v \in H_P^1(D)$ it holds
\begin{equation*}
  \langle (\hat{H} - \lambda)v, v \rangle \leq \frac{\epsilon^2}{2}\|\nabla v\|^2 + \|v_0\|_{L^{\infty}}\|v\|^2.
\end{equation*}
Meanwhile, let $\lambda$ be the minimal eigenvalue of $\hat{H}$. We take the decomposition for $v = (v, \psi)\psi + \tilde{\psi}$,
which implies that $\langle (\hat{H}-\lambda)v, v \rangle = (\tilde{\psi}, \tilde{\psi}) \ge 0$.
Hence, there exists a nonnegative constant $M$ such that for all $v\in H_P^1(D)$
\begin{equation}
  0 \leq \langle (\hat{H} - \lambda)v, v \rangle \leq M\|v\|_1^2.
    \label{equ:upper-lower-bound}
\end{equation}

Before the formal convergence estimate for the EVP is given, we consider the elliptic problem
\begin{equation}
    a(u, v) = f(v),
    \label{equ:elliptic-problem}
\end{equation}
where $a(u, v) = \frac{\epsilon^2}{2}(\nabla u, \nabla v) + (v_0u, v)$, which has been defined in \eqref{equ:weak-from}.
\begin{lemma}{\cite{sun2016finite}}
  \label{lem:L2-leq-H1-connection}
  Given $f \in L^2(D)$, let $u_H$ be the solution of
  \begin{equation*}
    a(u_H, v_H) = f(v_H),\quad \forall v_H \in V_H.
  \end{equation*}
  The numerical solution $u_H \in V_H$ satisfies
  \begin{equation}
    \|u - u_H\| \leq C\|u - u_H\|_1 \sup_{g\in L^2(D), \|g\| \neq 0} \left\{ \frac 1{\|g\|}\inf_{v\in V_H} \|\phi_g - v\| \right\},
  \end{equation}
  where, for every $g \in L^2(D)$, $\phi_g \in H_P^1(D)$ denotes the corresponding unique solution of the equation
  \begin{equation}
    \langle \hat{H}w, \phi_g \rangle := a(w, \phi_g) = (g, w), \text{ for all } w \in H_P^1(D).
    \label{equ:dual-problem-source-problem}
  \end{equation}
\end{lemma}
\begin{proof}
  By the Riesz Representation Theorem, we can define
  \begin{equation}
    \|w\| = \sup_{g\in L^2(D), g\neq 0} \frac{(g, w)}{\|g\|}.
    \label{equ:RRT-result}
  \end{equation}
  Letting $w = u - u_H$ in \eqref{equ:dual-problem-source-problem}, since $a(u-u_H, v_H) = 0$, we get
  \begin{equation*}
    (g, u-u_H) = a(u-u_H, \phi_g) = a(u-u_H, \phi_g-v_H) \leq C\|u-u_H\|_1 \|\phi_g-v_H\|_1.
  \end{equation*}
  It follows that $(g, u-u_H) \leq C\|u-u_H\|_1 \inf_{v_H \in V_H}\|\phi_g-v_H\|_1$.
  Then the duality argument \eqref{equ:RRT-result} implies
  \begin{equation*}
    \|u - u_H\| \leq C \|u - u_H\|_1 \sup_{g\in L^2(D), g \neq 0} \left\{ \inf_{v_H \in V_H}\frac{\|\phi_g-v_H\|_1}{\|g\|} \right\}.
  \end{equation*}
  Furthermore, since $\phi_g$ solves \eqref{equ:dual-problem-source-problem}, if $u \in H^r(D) \cap H_P^1(D)$ with $1 \leq r \leq 2$, we have
  \begin{equation*}
    \sup_{g\in L^2(D), g \neq 0} \left\{ \inf_{v_H \in V_H}\frac{\|\phi_g-v_H\|_1}{\|g\|} \right\} \leq CH^{s-1}.
  \end{equation*}
\end{proof}

Hence, for $w \in H^{-1}(D)$, denote $\Psi_w$ the unique solution of the adjoint problem
\begin{equation}
  \langle (\hat{H} - \lambda)\Psi_w, v\rangle = ( w, v ) \text{ for all } v \in \psi^{\bot},
  \label{equ:adjoint-problem-EVP}
\end{equation}
where $\Psi_w \in \psi^{\bot} := \{ v \in H_P^1(D) | (\psi, v) = 0\}$.
Since $\lambda$ is the minimal eigenvalue, there exists a non-negative constant $\beta$ such that
\begin{equation*}
  \beta \| v\|_1^2 \leq \left( (\hat{H} - \lambda)v, v \right).
\end{equation*}
We then get the existence and uniqueness of the solution to \eqref{equ:adjoint-problem-EVP} and the bound
\begin{equation}
  \|\Psi_w\|_1 \leq \beta^{-1}\| w \|.
\end{equation}

\begin{lemma}
  Assume that there exist a family $(V_H)_{H>0}$ of finite dimensional subspace of $H_P^1(D)$ such that
  \begin{equation}
    \min \{\|\psi - \psi_H\|_1, \psi_H\in V_H\} \xrightarrow[H\rightarrow 0^+]{} 0,
    \label{equ:subspace-convergence}
  \end{equation}
  and then it holds $\|\psi - \psi_H\|_1 \xrightarrow[H\rightarrow 0^+]{} 0$.
  The FEM approximation for the EVP satisfies
  \begin{equation}
    E(\psi_H) - E(\psi) \leq C\|\psi_H - \psi\|_1^2,
  \end{equation}
  and
  \begin{equation}
    |\lambda_H - \lambda| \leq C\|\psi_H - \psi\|_1^2,
    \label{equ:eigenvalue-bounded-by-H1}
  \end{equation}
  where $C$ is a constant $C$ and $H > 0$.
  Besides, there exists $H_0 > 0$ and $C > 0$ such that for all $0 < H < H_0$,
  \begin{equation}
    \|\psi_H - \psi\| \leq CH^{r-1}\|\psi_H - \psi\|_1.
  \end{equation}
\end{lemma}
\begin{proof}
  Let $P_H\psi \in V_H$ be such that
  \begin{equation*}
    \|\psi - P_H\psi\|_1 = \min\{ \|\psi - v_H\|_1, \forall v_H \in V_H \}.
  \end{equation*}
  From \eqref{equ:subspace-convergence}, we deduce that $(P_H\psi)_{H>0}$ converges to $\psi$ in $H_P^1(D)$ with $H\rightarrow 0$.

  Since $\lambda(\psi_H, \psi) = (\psi_H, \hat{H}\psi) = \lambda_H(\psi, \psi_H)$, we get
  \begin{gather*}
      \lambda_H - \lambda = \langle (\hat{H} - \lambda)(\psi_H - \psi), (\psi_H - \psi) \rangle, \\
      E(\psi_H) - E(\psi) = \frac 12 \langle \hat{H}\psi_H, \psi_H \rangle - \frac 12 \langle \hat{H}\psi, \psi \rangle
    = \frac 12\langle (\hat{H} - \lambda)(\psi_H - \psi), (\psi_H - \psi)\rangle.
  \end{gather*}
  According to \eqref{equ:upper-lower-bound}, we have
  \begin{equation*}
    E(\psi_H) - E(\psi) \leq \| \psi_H - \psi \|_1^2, \quad |\lambda_H - \lambda| \leq \| \psi_H - \psi \|_1^2.
  \end{equation*}
  Next, we estimate the error $\|\psi_H - \psi\|$. Let $\psi_H^*$ be the orthogonal projection of $\psi_H$ on the affine space $\{v \in L^2(D)| (\psi, v) = 1\}$. One has
  \begin{equation}
    \psi_H^* \in H_P^1(D), \psi_H^* - \psi \in \psi^{\bot}, \psi_H^* - \psi_H = \frac 12\|\psi_H - \psi\|^2\psi,
  \end{equation}
  from which we infer that
  \begin{align*}
    \|\psi_H - \psi\|^2 &= \int_D(\psi_H - \psi)(\psi_H^* - \psi) + \int_D(\psi_H - \psi)(\psi_H - \psi_H^*) \\
    &= \int_D(\psi_H - \psi)(\psi_H^* - \psi) + \frac 14\|\psi_H - \psi\|^4 \\
    &= \langle (\hat{H} - \lambda)\Psi_{\psi_H-\psi}, \psi_H^* - \psi\rangle + \frac 14\|\psi_H - \psi\|^4 \\
    &= \langle (\hat{H} - \lambda)\Psi_{\psi_H-\psi}, \psi_H^* - \psi_H\rangle + \langle (\hat{H} - \lambda)\Psi_{\psi_H-\psi}, \psi_H - \psi\rangle + \frac 14\|\psi_H - \psi\|^4 \\
    &= \langle (\hat{H} - \lambda)\Psi_{\psi_H-\psi}, \psi_H - \psi\rangle + \frac 14\|\psi_H - \psi\|^4.
  \end{align*}
  Therefore, for all $\Psi_H \in V_H$, it holds
  \begin{equation*}
    \|\psi_H - \psi\|^2 = \langle (\hat{H} - \lambda)(\psi_H - \psi), \Psi_{\psi_H-\psi} - \Psi_H \rangle + \langle (\hat{H} - \lambda)(\psi_H - \psi), \Psi_{H}\rangle + \frac 14\|\psi_H - \psi\|^4.
  \end{equation*}
  For the first term of the above equation, we obtain an estimate
  \begin{equation}
    \langle (\hat{H} - \lambda)(\psi_H - \psi), \Psi_{\psi_H-\psi} - \Psi_H \rangle \leq C \|\psi_H - \psi\|_1 \|\Psi_{\psi_H-\psi} - \Psi_H\|_1.
  \end{equation}
  Furthermore, let $\Psi_H \in V_H\cap \psi^{\bot}$, and we obtain
  \begin{align*}
    &\langle (\hat{H} - \lambda)(\psi_H - \psi), \Psi_{H}\rangle = \langle (\hat{H} - \lambda)\psi_H, \Psi_{H}\rangle - \langle (\hat{H} - \lambda) \psi, \Psi_{H}\rangle \\
    = &\langle (\hat{H} - \lambda)\psi_H, \Psi_{H}\rangle - 0
    = \left( (\lambda_{H} - \lambda)\psi_H, \Psi_{H}\right) - \left( (\lambda_{H} - \lambda)\psi, \Psi_{H}\right) \\
    = &(\lambda_{H} - \lambda)\left( (\psi_H - \psi), \Psi_{H} \right),
  \end{align*}
  which implies
  \begin{equation*}
    \left|\langle (\hat{H} - \lambda)(\psi_H - \psi), \Psi_{H}\rangle \right| \leq (\lambda_H - \lambda)\|\psi_H - \psi\|\|\Psi_{H}\|.
  \end{equation*}
  Then, for all $\Psi_H \in V_H\cap \psi^{\bot}$, we get
  \begin{equation*}
    \|\psi_H - \psi\|^2 \leq C \|\psi_H - \psi\|_1 \|\Psi_{\psi_H-\psi} - \Psi_H\|_1 + \|\psi_H - \psi\|_1^2 \|\psi_H - \psi\|\|\Psi_{H}\| + \frac 14\|\psi_H - \psi\|^4.
  \end{equation*}
  By \Cref{lem:L2-leq-H1-connection}, we have $\|\Psi_{\psi_H-\psi} - \Psi_H\|_1 \leq CH^{r-1} \|\psi_H-\psi\|$ for $\psi \in H^r(D)\cap H_P^1(D)$ with $1 \leq r \leq 2$. Hence, owing to \eqref{equ:subspace-convergence}, we can conclude that there exists $H_0 > 0$ and a positive constant $C$ such that for all $0 < H < H_0$,
  \begin{equation}
    \|\psi_H - \psi\| \leq CH^{r-1}\|\psi_H - \psi\|_1.
  \end{equation}
\end{proof}

Next, we estimate the MsFEM approximation error. Let $P_H$ be the classical $L^2$-projection onto $V_H$ and $W = ker(P_H) = \{v\in H_P^1(D)| P_H(v) = 0\}$ be the kernel space. There exists an orthogonal splitting $ H_P^1(D) = V_H \oplus W$,
in which $W$ captures the fine mesh details from $H_P^1(D)$ that are not captured by $V_H$. Similarity, denote
\begin{equation}
  V_{ms} = \{ v\in H_P^1(D) | a(v, w) = 0 \text{ for all } w\in W\},
\end{equation}
and wherein there is another orthogonal decomposition, namely
\begin{equation}
  H_P^1(D) = V_{ms} \oplus W.
\end{equation}
We then seek the eigenvalues and the eigenfunctions in $V_{ms}$ such that
\begin{equation}
  a(\psi_{ms}, \phi) = \lambda_{ms}(\psi_{ms}, \phi), \quad \forall \phi\in V_{ms}
  \label{equ:discretized-MsFEM-EVP_deterministic}
\end{equation}
with $\|\psi_{ms}\| = 1$.

We revisit the elliptic problem \eqref{equ:elliptic-problem}. Let $u \in H_P^1(D)$, and we have $u - u_{ms} \in W$, i.e., $a(u-u_{ms}, v) = 0$ for any $v \in V_{ms}$. Owing to this orthogonality,
\begin{equation}
  a(u_{ms}-u, w) = f(w), \quad \forall w \in W.
\end{equation}
Since $u_{ms} - u \in W \subset H_P^1(D)$, we have $P_H(u_{ms} - u) = 0$, which implies
\begin{equation}
  \|u_{ms} - u\| \leq \|u_{ms} - u - P_H(u_{ms} - u)\| \leq CH\|u_{ms} - u\|_1.
\end{equation}
Furthermore, let $\beta > 0$ denotes the coercivity constant of $a(\cdot, \cdot)$, and then the variational equation gives
\begin{equation*}
  \beta\|u_{ms} - u\|_1^2 \leq a(u_{ms} - u, u_{ms} - u) = f(u_{ms} - u).
\end{equation*}
Meanwhile, we also have
\begin{equation*}
  f(u_{ms} - u) = (f, u_{ms} - u) = (f - P_H(f), u_{ms} - u - P_H(u_{ms} - u)) \leq CH^3\|f\|_2\|u_{ms} - u\|_1.
\end{equation*}
These indicate
\begin{equation}
  \|u_{ms} - u\|_1 \leq CH^3\|f\|_2
\end{equation}
and
\begin{equation}
  \|u_{ms} - u\| \leq CH\|u_{ms} - u\|_1 \leq CH^4.
\end{equation}

Therefore, we obtain the error estimate of the MsFEM for the EVP of the deterministic Schr\"odinger operator.
\begin{theorem}
  \label{thm:convergence-rate-MsFEM}
  Let $\psi$ and $\psi_{ms}$ be the ground states of \eqref{equ:weak-from} and \eqref{equ:discretized-MsFEM-EVP_deterministic}, respectively. We have the approximation error
  \begin{equation}
    \|\psi_{ms} - \psi\|_1 \leq CH^3, \quad \|\psi_{ms} - \psi\| \leq CH^4,
    \label{equ:error-direct-MsFEM-EVP}
  \end{equation}
  and
  \begin{equation}
    |\lambda_{ms} - \lambda| \leq CH^6.
  \end{equation}
\end{theorem}

\begin{remark}
  The convergence rate $\mathcal{O}(H^6)$ of the minimal eigenvalue can also be obtained via a high-order interpolation in Theorem 4.1 in~\cite{CiCP-24-1073}. These super-convergence rates stated in the theorem above are also demonstrated in \cite{doi:10.1137/22M1516300}, in which the LOD is applied to the nonlinear Schr\"odinger equation with a cubic term. The comparable approximation accuracy is due to the analogous orthogonal decomposition of $H_P^1(D)$ employed by the LOD method and the MsFEM introduced in this work.
\end{remark}

\subsection{Dimension truncation error}
Here we denote $\lambda_s = \lambda_s(\boldsymbol{\omega}_s; \mathbf{0})$ the truncated eigenvalue and $\psi_s = \psi_s(\boldsymbol{\omega}_s; \mathbf{0})$ the truncated eigenfunction. The truncation error with respect to $s$ is described as in \Cref{thm:truncated-approximation}.
\begin{proposition}[Theorem 4.1, \cite{gilbert2019analysis}]
  \label{thm:truncated-approximation}
  Suppose that \cref{assump:assumption-for-potentials} holds. There exist constants $C_1$, $C_2$, $C_3$, $C_4 > 0$ such that for sufficiently large $s$ and for all $\boldsymbol{\omega} \in \Omega$, the truncation errors of the minimal eigenvalue and the ground state are bounded with
  \begin{equation}
    |\lambda(\boldsymbol{\omega}) - \lambda_s(\boldsymbol{\omega}_s)| \leq C_1s^{-1/p + 1}, \quad
    \|\psi(\boldsymbol{\omega}) - \psi_s(\boldsymbol{\omega}_s)\|_1 \leq C_2s^{-1/p + 1}.
  \end{equation}
  Furthermore, the weak truncation error is bounded by
  \begin{equation}
    |\mathds{E}_{\boldsymbol{\omega}}[\lambda - \lambda_s]| \leq C_3 s^{-2/p + 1},
  \end{equation}
  and for any continuous linear functional $\mathcal{G} \in L^2(D; \Omega)$, we have
  \begin{equation}
    |\mathds{E}_{\boldsymbol{\omega}}[\mathcal{G}(\psi) - \mathcal{G}(\psi_s)]| \leq C_4 s^{-2/p + 1}.
  \end{equation}
  Here $C_1$, $C_2$, $C_3$ and $C_4$ are independent of $s$ and $\boldsymbol{\omega}$.
\end{proposition}

\subsection{QMC error}

Given the regularity as in~\Cref{lem:parametric-regularity}, we derive the upper bound of the root-mean-square error for the qMC approximation.
\begin{proposition}
    [Theorem 4.2, \cite{gilbert2019analysis}]
  \label{thm:qMC-approximaation}
  Let $N \in \mathds{N}$ be prime, $\mathcal{G} \in L^2(D; \Omega)$. Suppose that \cref{assump:assumption-for-potentials} holds. The root-mean-square errors of a component-by-component generated randomly shifted lattice rule approximations of $\mathds{E}_{\boldsymbol{\omega}}[\lambda_s]$ and $\mathds{E}_{\boldsymbol{\omega}}[\mathcal{G}(\psi_s)]$ are bounded by
  \begin{equation}
    \sqrt{\mathds{E}_{\boldsymbol{\Delta}}\left[ |\mathds{E}_{\boldsymbol{\omega}}[\lambda_s] - Q_{N, s}\lambda_s|^2 \right]} \leq C_{1,\alpha} N^{-\alpha},
  \end{equation}
  and
  \begin{equation}
    \sqrt{\mathds{E}_{\boldsymbol{\Delta}}\left[ |\mathds{E}_{\boldsymbol{\omega}}[\mathcal{G}(\psi_s)] - Q_{N, s}\mathcal{G}(\psi_s)|^2 \right]} \leq C_{2,\alpha} N^{-\alpha},
  \end{equation}
  where
  \begin{equation}
    \label{equ:definition-alpha}
    \alpha = \left\{\begin{aligned}
      &1 - \delta, \quad \text{for arbitrary }\delta\in (0, \frac 12), \quad &p \in (0, \frac 23], \\
      &\frac 1p - \frac 12 &p \in (\frac 23, 1),
    \end{aligned}\right.
  \end{equation}
  and the constants $C_{1, \alpha}$ and $C_{2, \alpha}$ are independent of $s$.
\end{proposition}

Since a particular case of the elliptic problem, the linear Schr\"odinger operator, is considered in this work, the
proofs of \cref{thm:truncated-approximation} and \cref{thm:qMC-approximaation} are the same as those of Theorem 4.1 and Theorem 4.2 presented in \cite{gilbert2019analysis}.

\subsection{POD error}
Since the multiscale basis is approximated by the POD basis, the analysis begins with the estimation of basis function approximation. Here we consistently assume that the solutions of optimal problems \eqref{equ:optimal-problem-objective}-\eqref{equ:optimal-problem} exist and are bounded.

\begin{lemma}
  Let \cref{assump:assumption-for-potentials} hold and $\omega^1, \omega^2 \in \Omega$. The multiscale basis functions $\phi_i(\omega^1), \phi_i(\omega^2)$ are obtained by solving the optimal problem \eqref{equ:optimal-problem-objective}-\eqref{equ:optimal-problem} with random potentials $V(\omega^1)$ and $V(\omega^2)$, respectively. Then it holds that
  \begin{equation*}
    \|\phi_i(\omega^1) - \phi_i(\omega^2)\| \leq C\|V(\omega^1) - V(\omega^2)\|_{\infty} \|\phi(\omega^l)\|,
  \end{equation*}
  where $l = 1, 2$ and $i = 1, \cdots, N_H$.
  \label{lem:gap-msfem-basis-potential}
\end{lemma}
\begin{proof}
  The optimal problem \eqref{equ:optimal-problem-objective}-\eqref{equ:optimal-problem} can be equivalently formulated into a Karush-Kuhn-Tucker (KKT) equation. At $\mathbf{x}_i$, the corresponding KKT equation is
  \begin{equation*}
    \begin{pmatrix}
      G & -A^T \\
      A & O
    \end{pmatrix}\begin{pmatrix}
      \mathbf{c}_i \\ \boldsymbol{\lambda}_i
    \end{pmatrix} = \begin{pmatrix}
      \mathbf{0} \\ \mathbf{b}_i
    \end{pmatrix}
  \end{equation*}
  where $G$ is positive definite and
  \begin{equation*}
    \begin{pmatrix}
      G & -A^T \\
      A & O
    \end{pmatrix}^{-1} = \begin{pmatrix}
      G^{-1}-G^{-1}A^T(AG^{-1}A^T)^{-1}AG^{-1} & G^{-1}A^T(AG^{-1}A^T)^{-1} \\
      (G^{-1}A^T(AG^{-1}A^T)^{-1})^T  & -(AG^{-1}A^T)^{-1}
    \end{pmatrix}.
  \end{equation*}
  We therefore get the solution $\mathbf{c}_i = G^{-1}A^T(AG^{-1}A^T)^{-1}\mathbf{b}_i$. The matrix $G$ depends on the stochastic parameter, i.e.,
  \begin{equation*}
    \begin{pmatrix}
      G_1 & -A^T \\
      A & O
    \end{pmatrix}\begin{pmatrix}
      \mathbf{c}_i(\omega^1) \\ \boldsymbol{\lambda}_i(\omega^1)
    \end{pmatrix} = \begin{pmatrix}
      \mathbf{0} \\ \mathbf{b}_i
    \end{pmatrix}, \quad
    \begin{pmatrix}
      G_2 & -A^T \\
      A & O
    \end{pmatrix}\begin{pmatrix}
      \mathbf{c}_i(\omega^2) \\ \boldsymbol{\lambda}_i(\omega^2)
    \end{pmatrix} = \begin{pmatrix}
      \mathbf{0} \\ \mathbf{b}_i
    \end{pmatrix},
  \end{equation*}
  where $G_1 = G(\omega^1)$ and $G_2 = G(\omega^2)$.
  A straightforward derivation yields
  \begin{equation*}
    \begin{pmatrix}
      G_1 & -A^T \\
      A & O
    \end{pmatrix}\begin{pmatrix}
      \mathbf{c}_i(\omega^1) - \mathbf{c}_i(\omega^2)  \\ \boldsymbol{\lambda}_i(\omega^1) - \boldsymbol{\lambda}_i(\omega^2)
    \end{pmatrix} = \begin{pmatrix}
      (G_2 - G_1)\mathbf{c}_i(\omega^2) \\ \mathbf{0}
    \end{pmatrix},
  \end{equation*}
  which infers that
  \begin{equation}
    \mathbf{c}_i(\omega^1) - \mathbf{c}_i(\omega^2) = G_1^{-1} (G_2 - G_1)\mathbf{c}_i(\omega^2) - G_1^{-1}A^T(AG_1^{-1}A^T)^{-1}AG_1^{-1} (G_2 - G_1)\mathbf{c}_i(\omega^2).
    \label{equ:gap-basis-different-samples}
  \end{equation}

  Adopting the truncated expansion of random potentials yields
  \begin{equation*}
    G_{ij} = \frac{\epsilon^2}{2}(\nabla\phi_i^h, \nabla\phi_j^h) + (v_0\phi_i^h, \phi_j^h) + \sum_{k=1}^{s} \omega_k(v_k\phi_i^h, \phi_j^h).
  \end{equation*}
  Then we get
  \begin{equation*}
    \delta G_{ij} = G_{ij}(\omega^1) - G_{ij}(\omega^2) = \sum_{k=1}^{s} (\omega^1_k - \omega^2_k)(v_k\phi_i^h, \phi_j^h).
  \end{equation*}
  Since $\|v_k(\mathbf{x})\|_{\infty}$ is bounded, we have $|(v_k\phi_i^h, \phi_j^h)| \leq Ch^d$. Let $h \leq \epsilon$, we also have $G_{ij} \sim o(h^d)$. Consequently, for bounded potentials $V(\omega^1)$ and $V(\omega^2)$, it holds
  $$|(V(\omega^1) - V(\omega^2)\phi_i^h, \phi_j^h)| \leq C\|V(\omega^1) - V(\omega^2)\|_{\infty}(\phi_i^h, \phi_j^h).$$
  We then deduce that $\|G_1 - G_2\| \leq C\|V(\omega^1) - V(\omega^2)\|_{\infty}\|M^h\|$. Now go back to \eqref{equ:gap-basis-different-samples}, and we obtain
  \begin{align*}
    \|\mathbf{c}_i(\omega^1) - \mathbf{c}_i(\omega^2)\|
    &\leq \frac{\|G_1 - G_2\|\|\mathbf{c}_i(\omega^2)\|}{\|G_1\|}\left( 1 + \frac{\|A\|^2\|G_1^{-1}\|}{\|AG_1^{-1}A^T\|}\right) \\
    &\leq C\|V(\omega^1) - V(\omega^2)\|_{\infty}\|\mathbf{c}_i(\omega^2)\|\left( 1 + \frac{\|A\|^2\|G_1^{-1}\|}{\|AG_1^{-1}A^T\|}\right).
  \end{align*}
  Since $A_{ij} \sim o(h^d)$, there exists a positive constant $C$ such that
  \begin{equation*}
    \|\mathbf{c}_i(\omega^1) - \mathbf{c}_i(\omega^2)\| \leq C\|V(\omega^1) - V(\omega^2)\|_{\infty} \|\mathbf{c}_i(\omega^2)\|.
  \end{equation*}
  This bound holds uniformly for $i = 1, \cdots, N_h$. Denote $\Phi = (\phi_1^h, \cdots, \phi_{N_h}^h)$, and then $\phi_i(\omega^l) = \Phi\mathbf{c}_i(\omega^l)$ $(l = 1,2)$. Since finite-dimensional spaces are considered in this proof, we readily deduce that
  \begin{equation*}
    \|\phi_i(\omega^1) - \phi_i(\omega^2)\| \leq C\|V(\omega^1) - V(\omega^2)\|_{\infty} \|\phi_i(\omega^l)\|,
  \end{equation*}
  where $l = 1,2$ and $C$ is independent of potentials.
  This completes the proof.
\end{proof}


In the offline stage of \Cref{alg:MsFEM-POD-random-EVP},
for all $i = 1, \cdots, N_H$, we construct the reduced POD basis $\{\zeta_i^1(\mathbf{x}), \cdots, \zeta_i^{m_i}(\mathbf{x})\}$ with $m_i \ll Q$. According to the Proposition 1~\cite{Kunisch2001}, we have for all $\ell \leq m_i$
\begin{equation}
  \frac 1{Q}\sum_{j=1}^Q\left\| \tilde{\phi}_i(\omega^j) - \sum_{k=1}^{\ell}(\tilde{\phi}_i(\omega^j), \zeta_i^k)\zeta_i^k \right\|^2 = \sum_{\ell + 1}^{m_i}\sigma_k,
\end{equation}
which infers that there exists a constant $C$ such that for all $j \in \{1, \cdots, Q\}$,
\begin{equation}
  \left\| \tilde{\phi}_i(\omega^j) - \sum_{k=1}^{\ell}(\tilde{\phi}_i(\omega^j), \zeta_i^k)\zeta_p^k \right\|^2 \leq C \sum_{\ell + 1}^{m_i}\sigma_k.
  \label{equ:POD-error-random-potential}
\end{equation}

Next, we find the optimal approximation of the multiscale basis for random potentials in the space $V_{ms, i}^{pod} = span\{\zeta_i^0(\mathbf{x}), \zeta_i^1(\mathbf{x}), \cdots, \zeta_i^{m_i}(\mathbf{x})\}$ with the form
\begin{equation}
  \hat{\phi}_i(\mathbf{x}, \boldsymbol{\omega}) = \sum_{j=0}^{m_i}c_j(\boldsymbol{\omega})\zeta_i^j(\mathbf{x}).
\end{equation}
For any given stochastic variable $\boldsymbol{\omega}$, the optimal problems \eqref{equ:optimal-problem-objective}-\eqref{equ:optimal-problem} and \eqref{equ:optimal-problem2} can be equivalently written as
\begin{gather}
  \phi_i(\mathbf{x}, \boldsymbol{\omega}) = \argmin_{\phi \in H_P^1(D), (\phi, \phi_j^{H}) = \alpha\delta_{ij}} \frac{\epsilon^2}{2}\|\nabla\phi\|^2 + (V(\mathbf{x}, \boldsymbol{\omega})\phi, \phi),  \label{equ:MsFEM-equality-form} \\
  \hat{\phi}_i(\mathbf{x}, \boldsymbol{\omega}) = \argmin_{\phi \in V_{ms,i}^{pod}, (\phi, \phi_i^{H}) = \alpha} \frac{\epsilon^2}{2}\|\nabla\phi\|^2 + (V(\mathbf{x}, \boldsymbol{\omega})\phi, \phi).  \label{equ:MsFEM-POD-equality-form}
\end{gather}
Due to $V_{ms,i}^{pod} \subset H_P^1(D)$, we consider the optimal approximation problem
\begin{equation}
  \hat{\phi}_{i}(\mathbf{x}, \boldsymbol{\omega}) = \arginf_{\phi\in V_{ms,i}^{pod}, (\phi, \phi_i^H) = \alpha}\| \phi(\mathbf{x}, \boldsymbol{\omega}) - \phi_i(\mathbf{x}, \boldsymbol{\omega}) \|,
  \label{equ:optimal-approximation-POD-basis}
\end{equation}
and get the below lemma.

\begin{lemma}
  Given $\boldsymbol{\omega} \in \Omega$, let $\phi_i(\mathbf{x}, \boldsymbol{\omega})$ and $\hat{\phi}_i(\mathbf{x}, \boldsymbol{\omega})$ be the solutions of \eqref{equ:MsFEM-equality-form} and \eqref{equ:MsFEM-POD-equality-form}, respectively. For sufficiently small $h$, it holds
  \begin{equation}
    \|\phi_i(\mathbf{x}, \boldsymbol{\omega}) - \hat{\phi}_i(\mathbf{x}, \boldsymbol{\omega})\| \leq C\sqrt{\rho},
    \label{equ:POD-approximation-basis}
  \end{equation}
  where $i = 1, \cdots, N_H$ and $C$ is a constant independent of $\boldsymbol{\omega}$ and mesh size $h$.
  \label{prop:L2-gap-msfem-pod-basis}
\end{lemma}
\begin{proof}
  Denote $\Omega_0 = \{\omega^j\}_{j=1}^Q \subset \Omega$, and consider $\boldsymbol{\omega} \in \Omega_0$. According to \eqref{equ:POD-error-random-potential}, it is obvious that
  \begin{align*}
    \|\phi_i(\mathbf{x}, \boldsymbol{\omega}) - \hat{\phi}_i(\mathbf{x}, \boldsymbol{\omega})\| = \left\| \tilde{\phi}_i(\mathbf{x}, \boldsymbol{\omega}) -  \sum_{j=1}^{m_i}c_i^j(\boldsymbol{\omega})\zeta_i^j(\mathbf{x})\right\| \leq C\sqrt{\rho},
  \end{align*}
  where $c_i^j(\boldsymbol{\omega}) = (\tilde{\phi}_i, \zeta_i^j)$. We next consider $\boldsymbol{\omega} \in \Omega/\Omega_0$. For any $j \in \{1, \cdots, Q\}$, we have
  \begin{align*}
    &\|\phi_i(\mathbf{x}, \boldsymbol{\omega}) - \hat{\phi}_i(\mathbf{x}, \boldsymbol{\omega})\| \leq \|\phi_i(\mathbf{x}, \boldsymbol{\omega}) - \phi_i(\mathbf{x}, \omega^j)\| + \|\phi_i(\mathbf{x}, \omega^j) - \hat{\phi}_i(\mathbf{x}, \boldsymbol{\omega})\| \\
    \leq & C\|V(\mathbf{x}, \boldsymbol{\omega}) - V(\mathbf{x}, \omega^j)\|_{\infty}\|\phi_i(\mathbf{x}, \omega^j)\| + \left\| \tilde{\phi}_i(\mathbf{x}, \omega^j) - \sum_{k=1}^{m_i}c_k(\boldsymbol{\omega})\zeta_i^k(\mathbf{x}) \right\|.
  \end{align*}
  Owing to the boundedness of $V(\mathbf{x}, \boldsymbol{\omega})$ and $\|\phi_i(\mathbf{x}, \omega^j)\| \leq Ch^d$, it holds
  \begin{equation}
    \|\phi_i(\mathbf{x}, \boldsymbol{\omega}) - \hat{\phi}_i(\mathbf{x}, \boldsymbol{\omega})\| \leq C\|\boldsymbol{\omega} - \omega^j\|_{\infty}\| h^d + C\sqrt{\rho}.
  \end{equation}
  Let $h$ be sufficiently small, and we get \eqref{equ:POD-approximation-basis}. This completes the proof.
\end{proof}

Furthermore, consider the finite-dimensional representations
\begin{equation*}
  \phi_i(\mathbf{x}, \boldsymbol{\omega}) = \sum_{j=1}^{N_h}c_i^j(\boldsymbol{\omega})\phi_j^h, \quad \hat{\phi}_i(\mathbf{x}, \boldsymbol{\omega}) = \sum_{j=1}^{N_h}\hat{c}_i^j(\boldsymbol{\omega})\phi_j^h.
\end{equation*}
According to the $L^2$-bound in \Cref{prop:L2-gap-msfem-pod-basis}, there exists a constant $C$ such that
\begin{equation}
  \label{equ:H1-approximation-multiscale-basis-POD}
  \| \nabla\phi_i(\mathbf{x}, \boldsymbol{\omega}) - \nabla\hat{\phi}_i(\mathbf{x}, \boldsymbol{\omega}) \| \leq \frac{C\sqrt{\rho}}{h^2}.
\end{equation}

\begin{remark}
  For the $H^1$-error of the multiscale basis approximation, we can also consider the POD method in $H^1(D)$ (see example in \cite{refId0}), which shall provide a better estimation for \eqref{equ:H1-approximation-multiscale-basis-POD}.
\end{remark}

Next, we consider the approximation of the equation $a(u, v) = f(v)$ by MsFEM and POD-MsFEM.
Similar to \cite{MA2020112635}, we consider the algebraic equations constructed by the MsFEM and the MsFEM-POD, respectively. Denote $\mathbf{G}_{ij} = \frac{\epsilon^2}{2}(\nabla\phi_i, \nabla\phi_j) + (V(\mathbf{x}, \boldsymbol{\omega})\phi_i, \phi_j)$ and $\mathbf{f}_{i} = (f, \phi_i)$, and we get the algebraic equation discretized by the MsFEM as
\begin{equation}
  \mathbf{G} \mathbf{u} = \mathbf{f},
\end{equation}
The counterpart approximated by the MsFEM-POD method is
\begin{equation}
  \mathbf{\hat G} \mathbf{\hat{u}} = \mathbf{\hat f},
\end{equation}
where $\mathbf{\hat G}_{ij} = \frac{\epsilon^2}{2}(\nabla\hat{\phi}_i, \nabla\hat{\phi}_j) + (V(\mathbf{x}, \boldsymbol{\omega})\hat{\phi}_i, \hat{\phi}_j)$ and $\mathbf{\hat f}_{i} = (f, \hat{\phi}_i)$. Owing to \cref{assump:assumption-for-potentials}, we get
\begin{align*}
  |\mathbf{G}_{ij} - \mathbf{\hat G}_{ij}| &= \frac{\epsilon^2}{2h^2}|(\phi_i, \phi_j) - (\hat{\phi}_i, \hat{\phi}_j)| + |(V(\mathbf{x}, \boldsymbol{\omega})\phi_i, \phi_j) - (V(\mathbf{x}, \boldsymbol{\omega})\hat{\phi}_i, \hat{\phi}_j)| \\
  &\leq \left(\frac{\epsilon^2}{2h^2} + \|V(\mathbf{x}, \boldsymbol{\omega})\|_{\infty}\right)(\|\phi_i\| + \|\hat{\phi}_j\|)\sqrt{\rho} \\
  &\leq CH^{d/2}\sqrt{\rho}.
\end{align*}

Define \(\mathbf{E}\) as the error between \(\mathbf{G}\) and \(\mathbf{\hat{G}}\), i.e., \(\mathbf{E} = \mathbf{G} - \mathbf{\hat{G}}\), as well as \(\mathbf{e}_f\) as the error such that \(\mathbf{e}_f = \mathbf{f} - \mathbf{\hat{f}}\). We can see that $|\mathbf{e}_{f, i}| \leq \|f\|\sqrt{\rho}$.
Consequently, we obtain
\begin{equation}
  \| \mathbf{u} - \mathbf{\hat u} \| = \|\mathbf{G}^{-1}( \mathbf{e}_f - \mathbf{E}\mathbf{\hat u})\| \leq \frac{1}{\|G\|}(\|\mathbf{e}_f\| + \|\mathbf{E}\| \|\mathbf{\hat u}\|) \leq C_1\sqrt{\rho},
\end{equation}
where $C_1$ depends on the bounds of $\|f\|$, $\|\mathbf{\hat{u}}\|$, $\|\mathbf{G}\|$ and $H$. Since $u_{ms} = \sum_{i=1}^{N_H}u_i\phi_i$ and $u_{ms}^{pod} = \sum_{i=1}^{N_H}\hat{u}_i\hat{\phi}_i$, we further get
\begin{align}
  \|u_{ms} - u_{ms}^{pod}\| &\leq \left\| \sum_{i=1}^{N_H}u_i\phi_i - \sum_{i=1}^{N_H} {u}_i\hat{\phi}_i \right\| + \left\| \sum_{i=1}^{N_H} {u}_i\hat{\phi}_i - \sum_{i=1}^{N_H}\hat{u}_i\hat{\phi}_i \right\| \nonumber\\
  &\leq \|\mathbf{u}\|\max_{1\leq i \leq N_H}\| \phi_i - \hat{\phi}_i \| + \sqrt{\sum_{i=1}^{N_H}\|\hat{\phi}_i\|^2}\|\mathbf{u} - \mathbf{\hat u}\| \label{equ:approximation-L2-error-POD-elliptic} \\
  &\leq C_2\sqrt{\rho}, \nonumber
\end{align}
where $C_2$ depends on $H$ and $\|\mathbf{u}\|$. Meanwhile, we also have
\begin{align*}
  \|\nabla u_{ms} - \nabla u_{ms}^{pod}\| &\leq \left\| \sum_{i=1}^{N_H}u_i\nabla\phi_i - \sum_{i=1}^{N_H} {u}_i\nabla\hat{\phi}_i \right\| + \left\| \sum_{i=1}^{N_H} {u}_i\nabla\hat{\phi}_i - \sum_{i=1}^{N_H}\hat{u}_i\nabla\hat{\phi}_i \right\| \\
  &\leq \|\mathbf{u}\|\max_{1\leq i \leq N_H}\| \nabla\phi_i - \nabla\hat{\phi}_i \| + \sqrt{\sum_{i=1}^{N_H}\|\nabla\hat{\phi}_i\|^2}\|\mathbf{u} - \mathbf{\hat u}\| \\
  &\leq C_3\sqrt{\rho},
\end{align*}
where $C_3$ depends on $\|\mathbf{u}\|$, $h$ and $C_1$. Note that $\|\nabla\hat{\phi}_i\|$ are bounded due to the solvability of optimization problems. Therefore, there exists a constant $C$ such that
\begin{equation}
  \|u_{ms} - u_{ms}^{pod}\|_1 \leq C\sqrt{\rho}.
  \label{equ:approximation-H1-error-POD-elliptic}
\end{equation}

Next, consider the EVP approximated by the MsFEM-POD and MsFEM
\begin{equation}
  \mathcal{A}(\boldsymbol{\omega}; \psi_{ms}^{pod}, v) = \lambda_{ms}^{pod}(\psi_{ms}^{pod}, v), \;\forall v \in V_{ms}^{pod},
  \label{equ:discretization-MsFEM-POD}
\end{equation}
and
\begin{equation}
  \mathcal{A}(\boldsymbol{\omega}; \psi_{ms}, v) = \lambda_{ms}(\psi_{ms}, v), \;\forall v \in V_{ms}.
\end{equation}
A direct derivation similar to \eqref{equ:approximation-L2-error-POD-elliptic} and \eqref{equ:approximation-H1-error-POD-elliptic} yields
\begin{equation}
    \|\psi_{ms}^{pod} - \psi_{ms}\|, \|\psi_{ms}^{pod} - \psi_{ms}\|_1 \leq C\sqrt{\rho}.
    \label{equ:error-MsFEM-MsFEM-POD-for-EVP}
\end{equation}

The approximation error of the MsFEM-POD for the EVP \eqref{equ:weak-form-random-EVP} is estimated as the following theorem.
\begin{theorem}
  \label{thm:error-MsFEM-POD}
  Let $\psi_{ms}^{pod}$ and $\lambda_{ms}^{pod}$ be the solution of the discretized form ~\eqref{equ:discretization-MsFEM-POD}, we have
  \begin{equation}
    \|\psi_{ms}^{pod} - \psi\|_1 \leq C(H^3 + \sqrt{\rho}), \quad \|\psi_{ms}^{pod} - \psi\| \leq C(H^4 + \sqrt{\rho}),
    \label{equ:error-eigenstate-msfem-pod-EVP}
  \end{equation}
  and
  \begin{equation}
    |\lambda_{ms}^{pod} - \lambda| \leq C(H^6 + \rho).
  \end{equation}
\end{theorem}
\begin{proof}
  Since $\|\psi_{ms}^{pod} - \psi\|_1 \leq \|\psi_{ms}^{pod} - \psi_{ms}\|_1 + \|\psi_{ms} - \psi\|_1$ and $\|\psi_{ms}^{pod} - \psi\| \leq \|\psi_{ms}^{pod} - \psi_{ms}\| + \|\psi_{ms} - \psi\|$.  A combination of \eqref{equ:error-direct-MsFEM-EVP} and \eqref{equ:error-MsFEM-MsFEM-POD-for-EVP} yields the error bounds in \eqref{equ:error-eigenstate-msfem-pod-EVP}. Additionally, an application of \eqref{equ:eigenvalue-bounded-by-H1} yields
  \begin{align*}
    |\lambda_{ms}^{pod} - \lambda| &\leq |\lambda_{ms}^{pod} - \lambda_{ms}| + |\lambda_{ms} - \lambda| \\
    &\leq \|\psi_{ms}^{pod} - \psi_{ms}\|_1^2 + \|\psi_{ms} - \psi\|_1^2\\
    &\leq C(H^6 + \rho).
  \end{align*}
  These complete the proof.
\end{proof}

\subsection{Total error}
In the above, we outline the error of MsFEM approximation error in physic space, the truncation error of the model, the qMC approximation error, and the MsFEM-POD approximation error. Combine these errors and we get the following theorem for the total error.
\begin{theorem}
  Suppose \cref{assump:assumption-for-potentials} holds, $s \in \mathds{N}$, $N \in \mathds{N}$ be prime and ${z} \in \mathds{N}^{s}$ be a generating vector constructed using the component-by-component algorithm with weights. The root-mean-square error with respect to the random shift $\boldsymbol{\Delta} \in [0, 1]^s$, of the MsFEM-POD with the qMC method for the minimal eigenvalue $\lambda$ is bounded by
  \begin{equation}
    \sqrt{\mathds{E}_{\boldsymbol{\Delta}}\left[ |\mathds{E}_{\boldsymbol{\omega}}[\lambda] - Q_{N, s}\lambda_{s,ms}^{pod}|^2 \right]} \leq C\left( H^6 + \rho + s^{-2/p+1} + N^{-\alpha} \right).
  \end{equation}
  Meanwhile, for any $\mathcal{G} \in L^2(D; \Omega)$ applying to the ground state $\psi$, the counterpart error approximation of its mean is bounded by
  \begin{equation}
    \sqrt{\mathds{E}_{\boldsymbol{\Delta}}\left[ |\mathds{E}_{\boldsymbol{\omega}}[\mathcal{G}(\psi)] - Q_{N, s}\mathcal{G}(\psi_{s,ms}^{pod})|^2 \right]} \leq C\left( H^3 + \sqrt{\rho} + s^{-2/p+1} + N^{-\alpha} \right).
  \end{equation}
  Here $\alpha$ is defined as the \eqref{equ:definition-alpha}.
  \label{thm:total-error}
\end{theorem}

\section{Numerical experiments}
\label{sec:experiments}
In this section, we numerically check the convergence rates of the proposed method. After that, we investigate the localization of the eigenstates for the Schr\"odinger operator with spatially random potentials. In all cases, we compute the eigenvalues using MATLAB's $eigs$ with the option $smallestabs$.

\subsection{Superconvergence of the MsFEM discretization}
The 1D double-well potential and 2D checkboard potential are adopted to verify the superconvergence rates of the MsFEM method. In these experiments, we fix $\epsilon = 1$, and calculate the reference solution $(\lambda_{ref,l}, \psi_{ref,l})$ ($l = 1, \cdots, 5$) by the FEM with mesh size $h$.

\begin{example}
  Consider the 1D double-well potential $v_0(x) = (x^2 - 4)^2$ over the domain $D = [-4, 4]$. We fix $h = 1/256$ and vary $N_H = 8, 16, 32, 64$ and record the errors $| \lambda_{ref,l} - \lambda_{ms,l} |$, $\| \psi_{ref,l} - \psi_{ms,l} \|$ and $\| \psi_{ref,l} - \psi_{ms,l} \|_1$. As shown in \Cref{fig:comparsion-FEM_MsFEM}, the second-order convergence rates of FEM approximation and the superconvergence rates of the MsFEM approximation are depicted. In this experiment, the minimal eigenvalue and the ground state are calculated.
  \begin{figure}[htbp]
    \centering
    \subfloat[Eigenvalue.]{\includegraphics[width=1.7in]{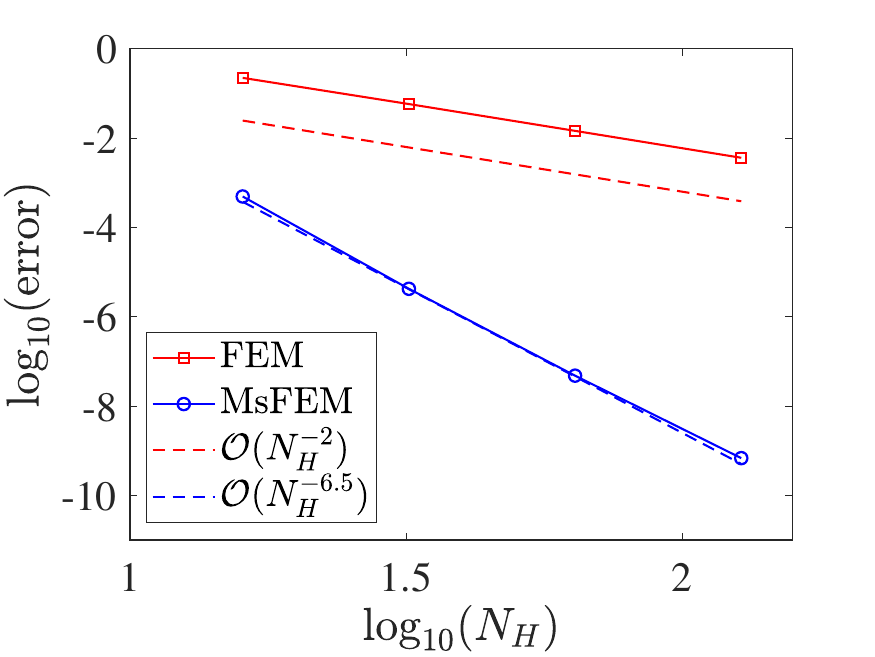}}
    \subfloat[$L^2$-error.]{\includegraphics[width=1.7in]{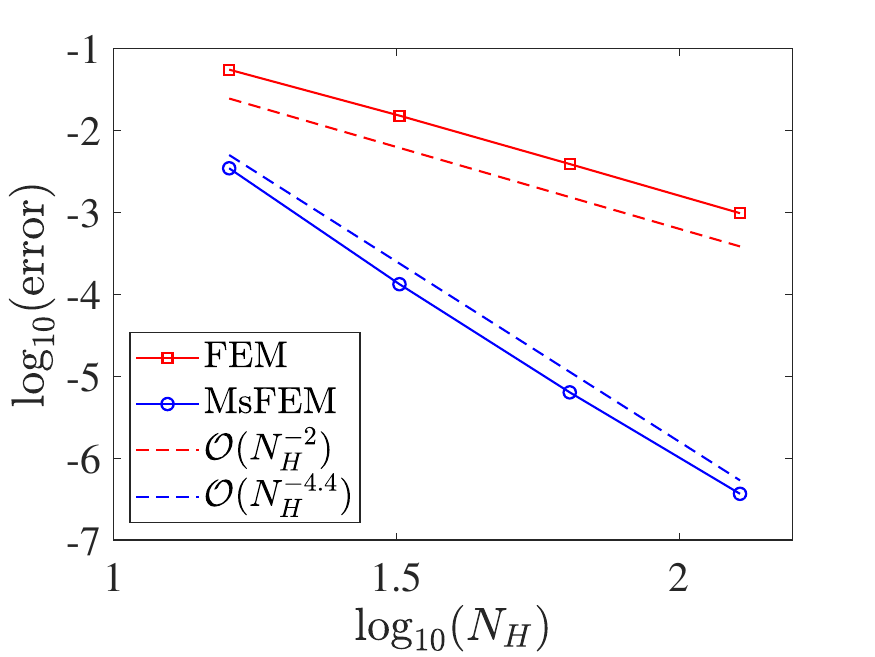}}
    \subfloat[$H^1$-error.]{\includegraphics[width=1.7in]{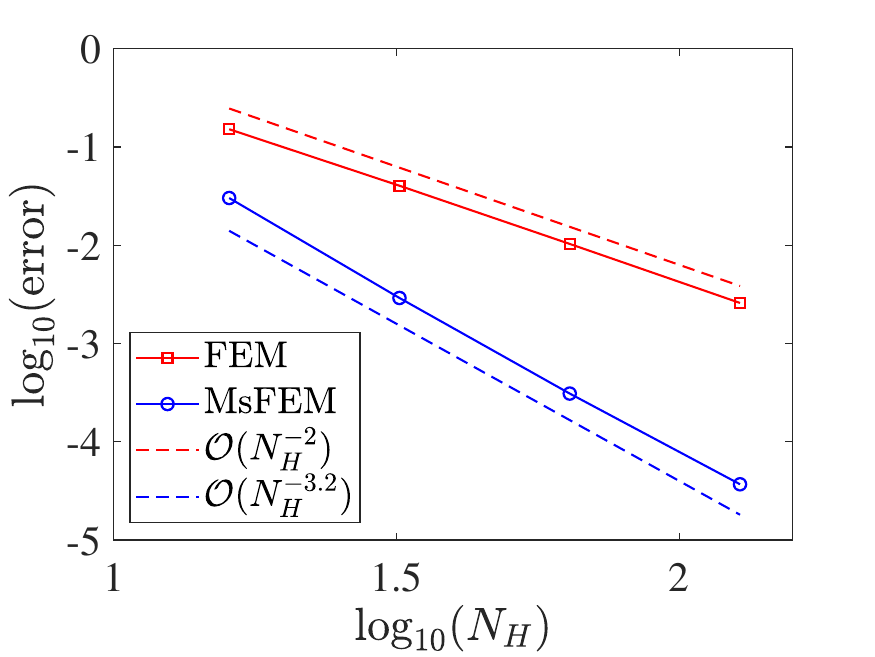}}
    \caption{Numerical convergence rates of the FEM and MsFEM approximation for the EVP of the Schr\"odinger operator with the 1D double-well potential.}
    \label{fig:comparsion-FEM_MsFEM}
  \end{figure}

  Furthermore, we check the approximation error of the MsFEM method for the first five eigenvalues and the corresponding eigenfunctions. Numerical results are depicted in \Cref{tab:5eigenvalues-convergence-double-well}, \Cref{tab:5eigenstates-L2-convergence-double-well} and \Cref{tab:5eigenstates-H1-convergence-double-well}. In \Cref{tab:5eigenstates-L2-convergence-double-well} and \Cref{tab:5eigenstates-H1-convergence-double-well}, the convergence rates of $\| \psi_{ref} - \psi_{ms} \|$ and $\| \psi_{ref} - \psi_{ms} \|_1$ are slight worse than the results in \Cref{fig:comparsion-FEM_MsFEM}. This difference is due to the inclusion of a coarse grid of $N_H = 8$ in both tables.
  \begin{table}[htbp]
    \centering
    \caption{Numerical convergence rates of the error $| \lambda_{ref,l} - \lambda_{ms,l} |$ ($l = 1, \cdots, 5$).}
    \begin{tabular}{||c|c|c|c|c|c||}
      \hline
      $\lambda_h^l$ & $N_H = 16$ & $N_H = 32$ & $N_H = 64$ & $N_H = 128$ & order \\
      \hline
      2.762420126423838 & 4.9166e-04 & 4.2144e-06 & 4.7839e-08 & 6.8088e-10 & -6.48 \\
      2.762436019658617 & 4.9329e-04 & 4.2143e-06 & 4.7835e-08 & 6.8081e-10 & -6.48 \\
      7.988965439736671 & 3.4864e-02 & 1.6993e-04 & 1.7874e-06 & 2.4865e-08 & -6.78 \\
      7.991063271042746 & 3.5019e-02 & 1.7032e-04 & 1.7901e-06 & 2.4897e-08 & -6.78 \\
      12.596293528481384 & 1.6578e-01 & 9.4183e-04 & 9.1390e-06 & 1.2373e-07 & -6.77 \\
      \hline
    \end{tabular}
    \label{tab:5eigenvalues-convergence-double-well}
  \end{table}
  \begin{table}[htbp]
    \centering
    \caption{Numerical convergence rates of the error $\| \psi_{ref,l} - \psi_{ms,l} \|$ ($l = 1, \cdots, 5$).}
    \begin{tabular}{||c|c|c|c|c|c|c||}
      \hline
      $l$ & $N_H = 8$ & $N_H = 16$ & $N_H = 32$ & $N_H = 64$ & $N_H = 128$ & order \\
      \hline
      1 & 1.2924e-02 & 3.4672e-03 & 1.3316e-04 & 6.3567e-06 & 3.6648e-07 & -3.93 \\
      2 & 1.4693e-02 & 3.4861e-03 & 1.3316e-04 & 6.3565e-06 & 3.6646e-07 & -3.97 \\
      3 & 4.1625e-01 & 4.0202e-02 & 8.9026e-04 & 3.9427e-05 & 2.2222e-06 & -4.50 \\
      4 & 4.6330e-01 & 4.0305e-02 & 8.9142e-04 & 3.9462e-05 & 2.2237e-06 & -4.53 \\
      5 & 7.5693e-01 & 1.0805e-01 & 2.2233e-03 & 9.0568e-05 & 4.9787e-06 & -4.46 \\
      \hline
    \end{tabular}
    \label{tab:5eigenstates-L2-convergence-double-well}
  \end{table}
  \begin{table}[htbp]
    \centering
    \caption{Numerical convergence rates of the error $\| \psi_{ref,l} - \psi_{ms,l} \|_1$ ($l = 1, \cdots, 5$).}
    \begin{tabular}{||c|c|c|c|c|c|c||}
      \hline
      $l$ & $N_H = 8$ & $N_H = 16$ & $N_H = 32$ & $N_H = 64$ & $N_H = 128$ & order \\
      \hline
      1 & 5.8489e-02 & 3.0251e-02 & 2.9061e-03 & 3.0933e-04 & 3.6920e-05 & -2.79 \\
      2 & 6.1666e-02 & 3.0283e-02 & 2.9061e-03 & 3.0931e-04 & 3.6918e-05 & -2.80 \\
      3 & 1.6051e-00 & 2.9316e-01 & 1.8688e-02 & 1.8953e-03 & 2.2312e-04 & -3.29 \\
      4 & 1.6491e-00 & 2.9374e-01 & 1.8709e-02 & 1.8968e-03 & 2.2327e-04 & -3.30 \\
      5 & 2.7166e-00 & 7.2761e-01 & 4.4559e-02 & 4.2944e-03 & 4.9794e-04 & -3.22 \\
      \hline
    \end{tabular}
    \label{tab:5eigenstates-H1-convergence-double-well}
  \end{table}
\end{example}

\begin{example}
  In this case, we adopt a checkerboard potential as depicted in \Cref{fig:2d-checkboard-convergence}(A). Over the domain $D = [-0.5, 0.5]^2$, the potential is set to a checkboard with squares of size $2^{-4}$, which results in $16\times 16$ squares. The values of sub-squares alternate between 0 and 2. We then calculate the reference solution with a uniform mesh size $h = 1/512$. Here we check the convergence rates of the minimal eigenvalue and the ground state, and the results are shown in \Cref{fig:2d-checkboard-convergence}.

  It is shown that for the discontinuous potential, both the FEM and MsFEM manage to retain near-optimal convergence of the minimal eigenvalue. However, for the ground state calculations, the MsFEM successfully preserves the convergence rates while the FEM fails. This showcases the superior resilience of the MsFEM to approximate eigenfunctions for discontinuous potentials.
  \begin{figure}[htbp]
    \centering
    \subfloat[Potential.]{\includegraphics[width=1.6in]{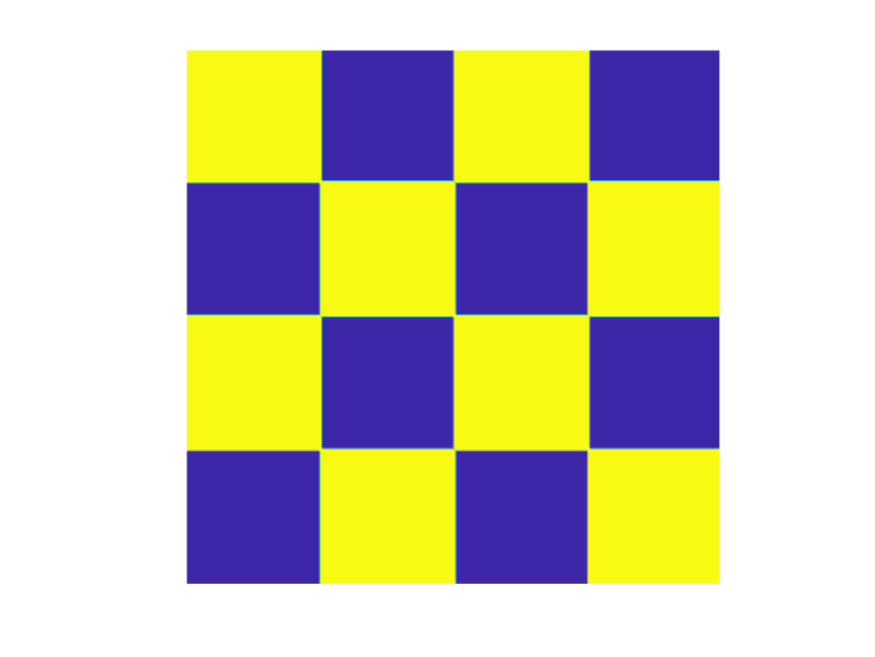}}
    \subfloat[Eigenvalue.]{\includegraphics[width=1.6in]{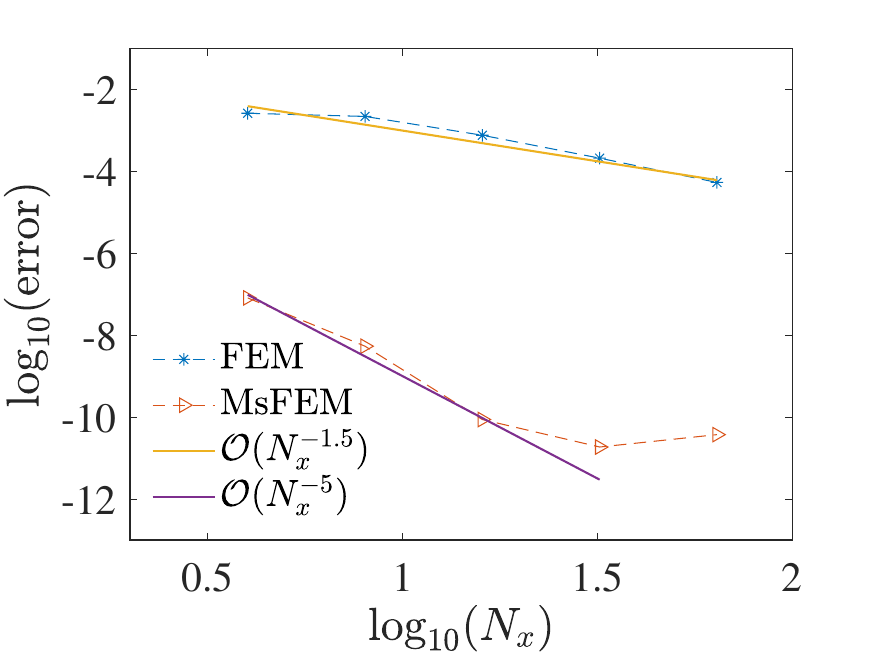}}
    \subfloat[Eigenfunction.]{\includegraphics[width=1.6in]{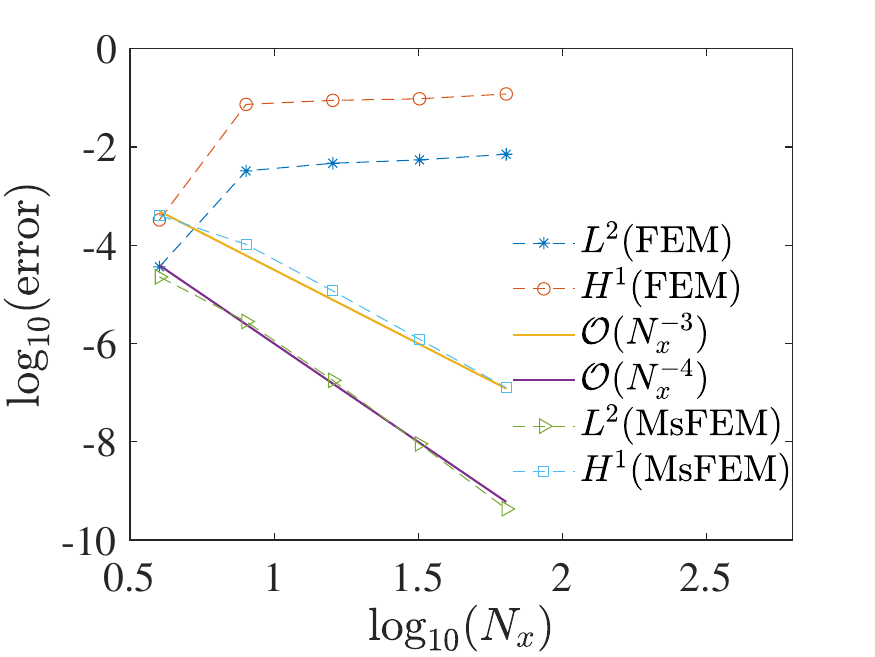}}
    \caption{The checkboard potential and the numerical convergence rates of the FEM and MsFEM methods.}
    \label{fig:2d-checkboard-convergence}
  \end{figure}
\end{example}

\subsection{Random potentials}
Next, we consider the parameterized potentials
\begin{equation}
    V(x, \boldsymbol{\omega}_s) = 1.0 + \sum_{j=1}^s \frac{\sin(j\pi x)}{1+(j\pi)^q}\omega_j,
    \label{equ:1D-random-potential}
\end{equation}
where $q$ controls the decaying rates of the high-frequency components.
For $q > 1$, we have for all $j \in \mathds{N}$, $\|v_j\|_{\infty} = \frac{1}{1+(j\pi)^q} < \frac{1}{(j\pi)^q}$ and hence $\sum_{j=1}^{\infty}\|v_j\|_{\infty} < \zeta(q)/\pi^q$. In turn, the value of $p$ in the \Cref{thm:truncated-approximation} can be in the interval $(1/q, 1)$.

The reference solutions are computed by
\begin{equation*}
  \mathds{E}[\lambda_k] = \frac{1}{N}\sum_{i=1}^{N}\lambda_k(\omega^i), \quad  \mathds{E}[\psi_k] = \frac{1}{N}\sum_{i=1}^{N}\psi_k(\omega^i),
\end{equation*}
where $(\lambda_k, \psi_k)$ are the FEM solution on a fine mesh. The empirical expectations of numerical solutions $(\mathds{E}[\lambda_{ms,k}], \mathds{E}[\psi_{ms,k}])$ are calculated similarly. Since the convergence rate of eigenvalues will be mainly concerned, we define the absolute error $$\mathrm{error}_k = | \mathds{E}[\lambda_{ms,k}] - \mathds{E}[\lambda_k] |,$$ where "$\mathrm{error}$" specifically represent the case of $k = 1$.

\begin{example}[Estimation of sample size for the POD basis.]
  \label{eg:estimation-Q}
  In the online stage, the multiscale basis associated with the random potentials is approximated by the POD basis. The samples for constructing the POD basis are crucial to the quality of the reduced basis. In this example, we choose different numbers of qMC and MC samples and record the error as the number of samples varies, to determine the appropriate number of random samples. We fix $q = 0$ and $N = 4000$ to generate the random potentials. For the 1D case, the coarse mesh size is $H = \frac 1{16}$, and we set $s = 64$ and compute the reference solution by the FEM with $N_h = 2048$ over the interval $[-1, 1]$. For the 2D case, the coarse mesh size is $H = \frac 1{32}$, and we set $s = 8$ and compute the reference solution by the FEM with $N_h = 128$ over the domain $[-\frac 12, \frac 12]^2$. In \Cref{tab:error-of-vary-Q}, we record the errors as the sampling number $Q$ varies. The results show that when $Q$ is of order 100, the qMC sample provides the best approximation.
  \begin{table}[htbp]
    \centering
    \caption{The error of the MsFEM-POD method with different sampling numbers in the offline stage.}
    \begin{tabular}{||c|c|c|c|c|c|c||}
      \hline
      $Q$ & 10 & 50 & 100 & 200 & 1000 \\
      \hline
      qMC, 1D & 2.0615e-03 & 1.0235e-03 & 1.3442e-03 & 1.2612e-03 & 1.1749e-03 \\
      MC, 1D & 2.8781e-03 & 1.8927e-03 & 1.8767e-04 & 1.6333e-03 & 1.2932e-03 \\
      \hline
      qMC, 2D & 1.6164e-04 & 3.6407e-04 & \textbf{7.3920e-07} & \textbf{7.1088e-07} & - \\
      MC, 2D & 1.1322e-04 & 3.7447e-04 & 3.7448e-04 & 3.7453e-04 & - \\
      \hline
    \end{tabular}
    \label{tab:error-of-vary-Q}
  \end{table}

  In addition, we plot the basis functions constructed by the optimal problems \eqref{equ:optimal-problem-objective}-\eqref{equ:optimal-problem} and \eqref{equ:optimal-problem2}, respectively, where 200 qMC samples are generated to construct the POD basis. We test the potentials parameterized by the random samples chosen in $\Omega_0$ and $\Omega/\Omega_0$, respectively. As shown in \Cref{fig:1d_basis_functions}, we get the accurate multiscale basis by solving the reduced optimal problems \eqref{equ:optimal-problem2}.
\begin{figure}[htbp]
  \centering
  \subfloat[1D basis functions and the error distribution over ${[-1, 1]}$.]{\includegraphics[width=1.6in]{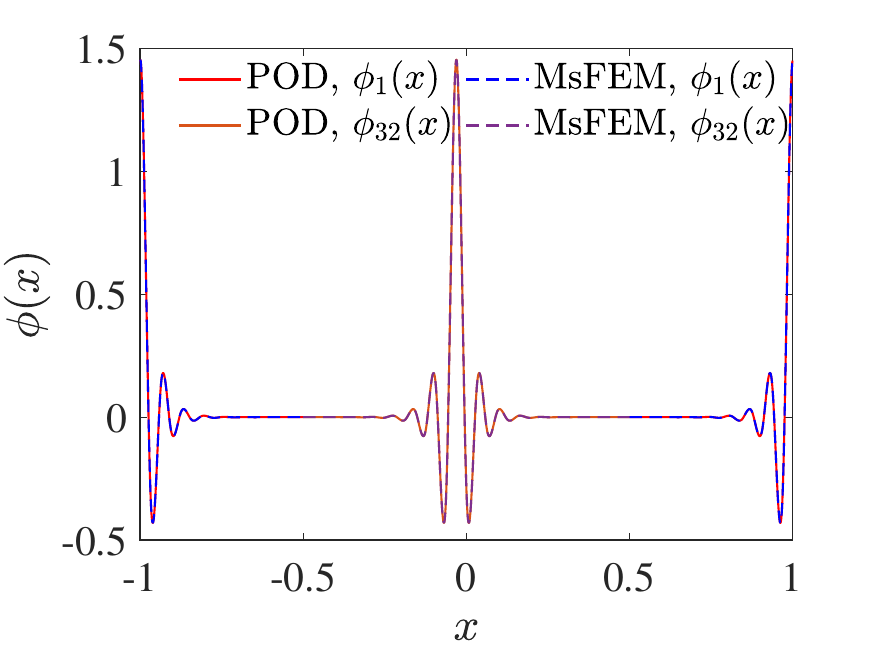}\includegraphics[width=1.6in]{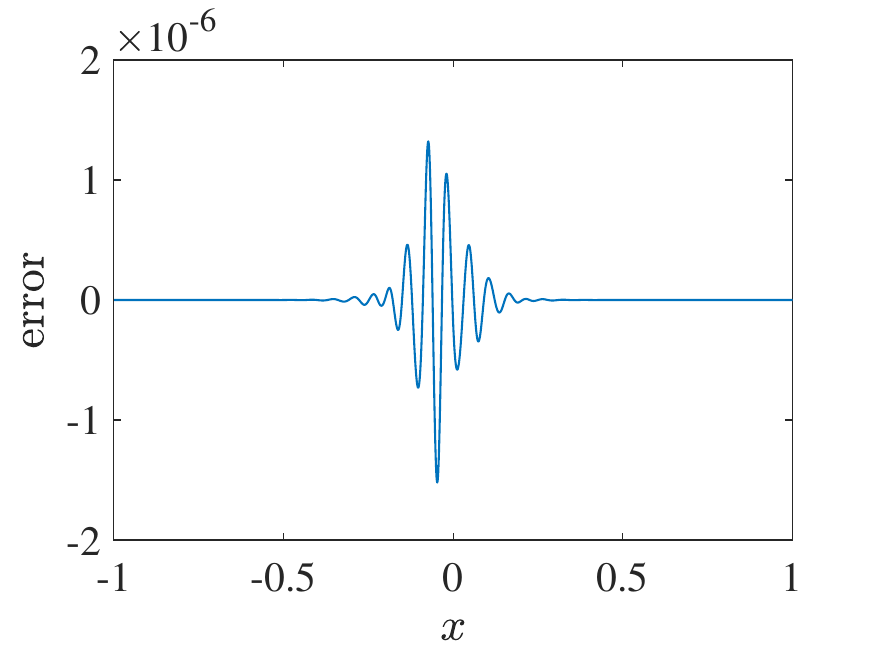}\includegraphics[width=1.6in]{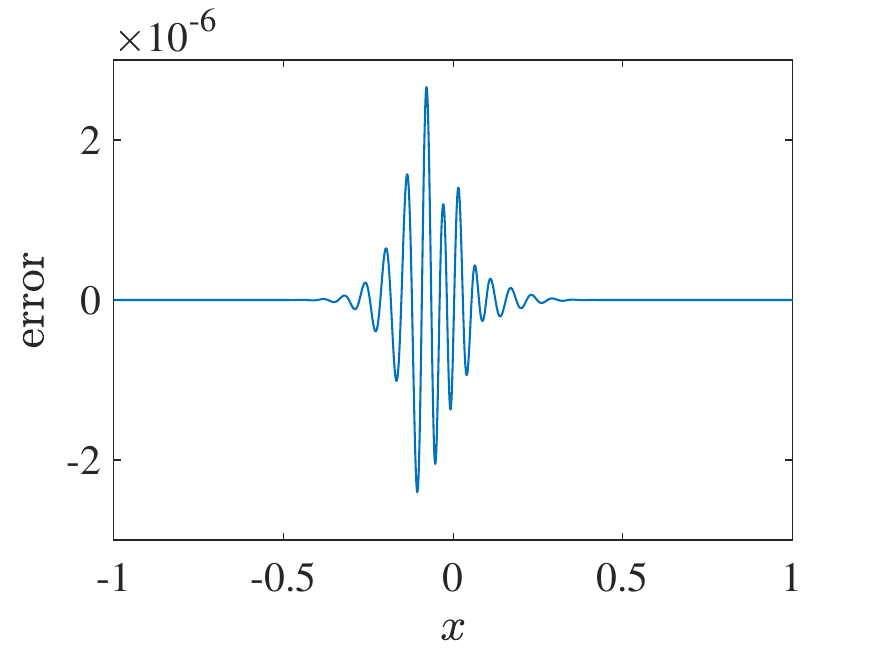}}\\
  \subfloat[2D basis functions and the error distribution over ${[-\frac 12, \frac 12]^2}$.]{\includegraphics[width=1.6in]{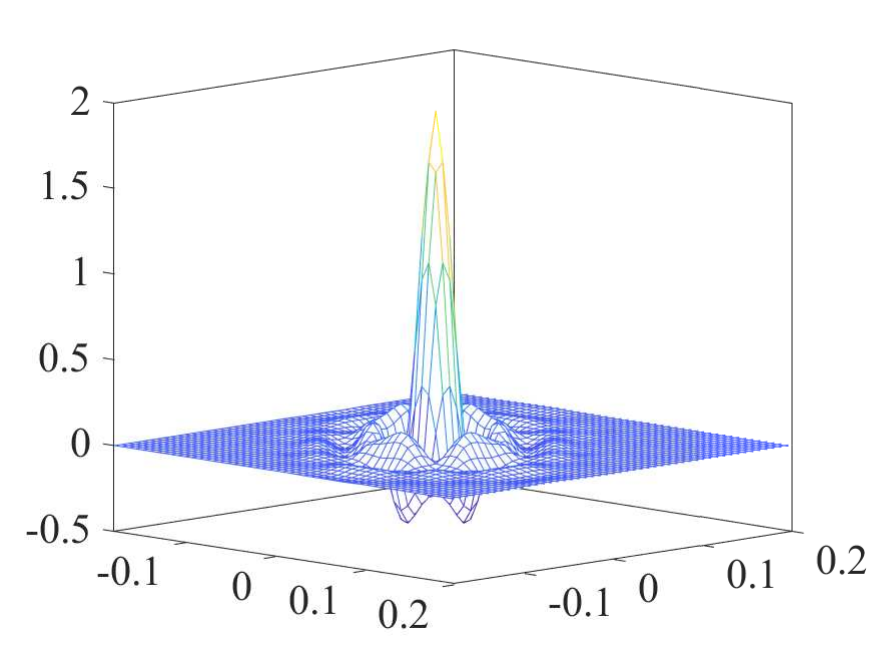}\includegraphics[width=1.6in]{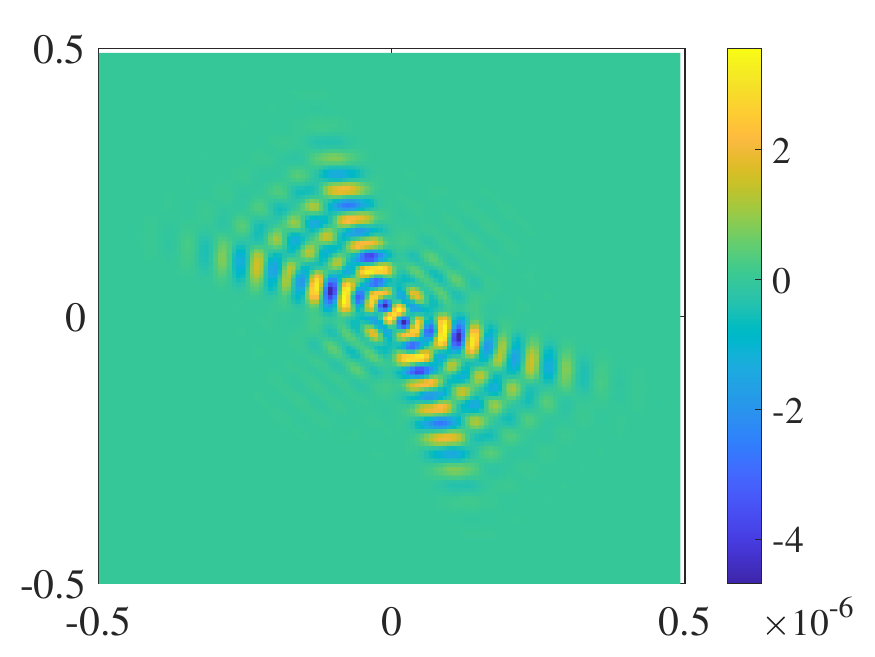}\includegraphics[width=1.6in]{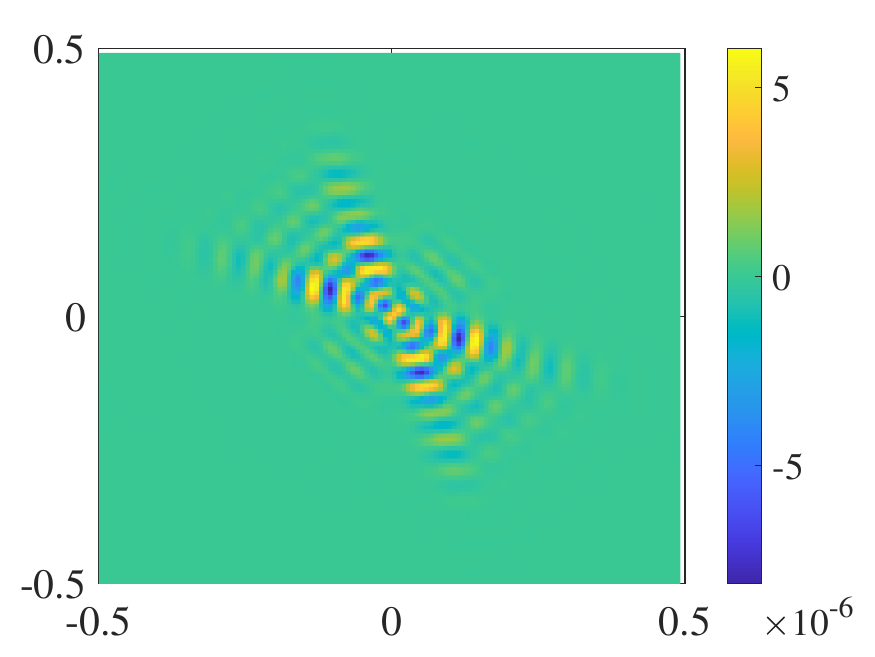}}\\
  \caption{The basis functions and the error between the multiscale basis solved by \eqref{equ:optimal-problem-objective}-\eqref{equ:optimal-problem} and \eqref{equ:optimal-problem2}. 1st column: sketches of basis functions. 2nd column: the case for $\omega \in \Omega_0$. 3rd column: the case for $\omega \in \Omega/\Omega_0$.}
  \label{fig:1d_basis_functions}
\end{figure}
\end{example}

Next, we check the convergence of the proposed MsFEM-POD method. We fix $Q = 200$ and $m_i = 3$ for all $i = 1, \cdots, N_H$ in the rest of examples. Notice that the POD error $\rho$ is not discussed here. Interested readers can refer to \cite{Kunisch2001} for more details.

\begin{example}
  In this example, we check the convergence rate of the MsFEM-POD with respect to $s$. Over the interval $[-1, 1]$, we take $N_h = 2048$ and $s = 512$ to generate the reference solution, where $8000$ qMC samples are generated. For the MsFEM and MsFEM-POD methods, we set $N_H = 64$. As shown in \Cref{fig:1d-convergence-wrt-s}, we record the error as varies $s = 2, 4, 8, 16, 32, 64, 128, 256$. Here different values $q = \frac 43, 3$ are tested.
  When $q = 3$, the reference solution is $\lambda = 0.985033892103644$, and the solution computed by the MsFEM is $0.999475730933365$. A significant error $5.8349$e-09 is then produced, which can be observed for $s \ge 16$ as in \Cref{fig:1d-convergence-wrt-s}. Similar errors can be observed with $q = \frac 43$. Besides, the POD error is also depicted when $q = \frac 43$ and $s = 256$.
  \begin{figure}[htbp]
    \centering
    \includegraphics[width=2.in]{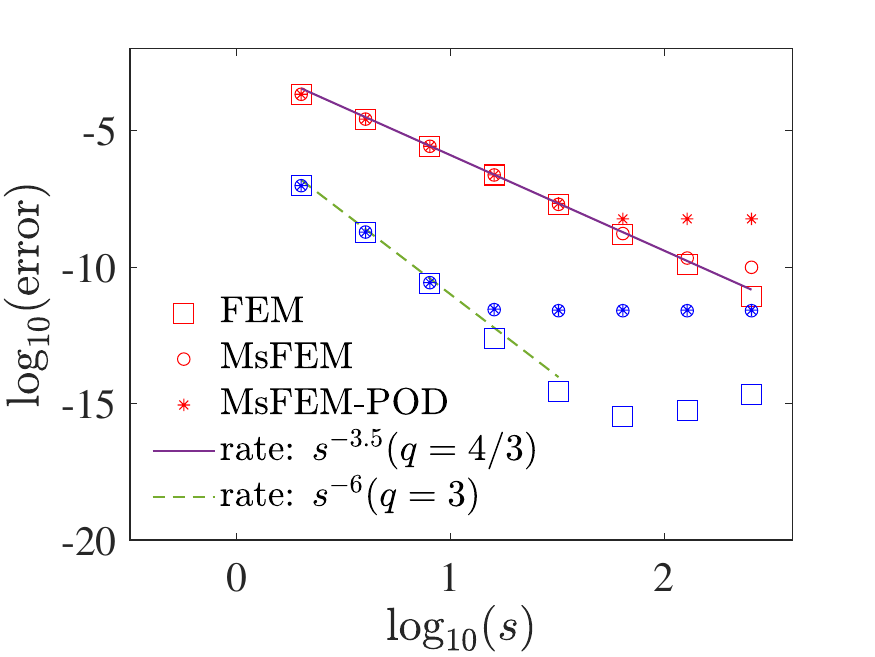}
    \caption{Numerical convergence rates with respect to $s$. The red and blue symbols denote the results corresponding to $q = 4/3$ and $q = 3$, respectively.}
    \label{fig:1d-convergence-wrt-s}
  \end{figure}

  Next, we verify the convergence of MsFEM-POD in the physical space. The reference solution is computed by the FEM with $q = \frac 43$, $s = 8$, $N = 8000$ and $N_h = 2048$. We vary $H = \frac 12, \frac 14, \cdots, \frac 1{64}$, and compare the convergence rates of FEM and MsFEM-POD as in~\Cref{fig:1D-convergence-CPU-time}(A). Meanwhile, the corresponding CPU time is also compared in \Cref{fig:1D-convergence-CPU-time}(B). The results demonstrate that the MsFEM-POD method offers an efficient approach for solving this class of random EVP.
  \begin{figure}[htbp]
    \centering
    \subfloat[Convergence rate.]{\includegraphics[width=2.in]{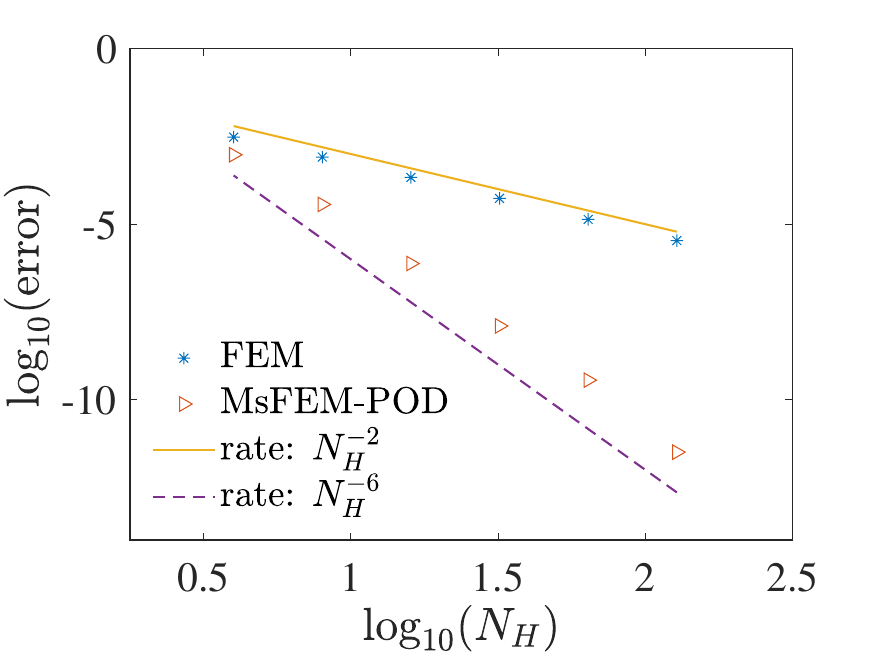}}
    \subfloat[CPU time.]{\includegraphics[width=2.in]{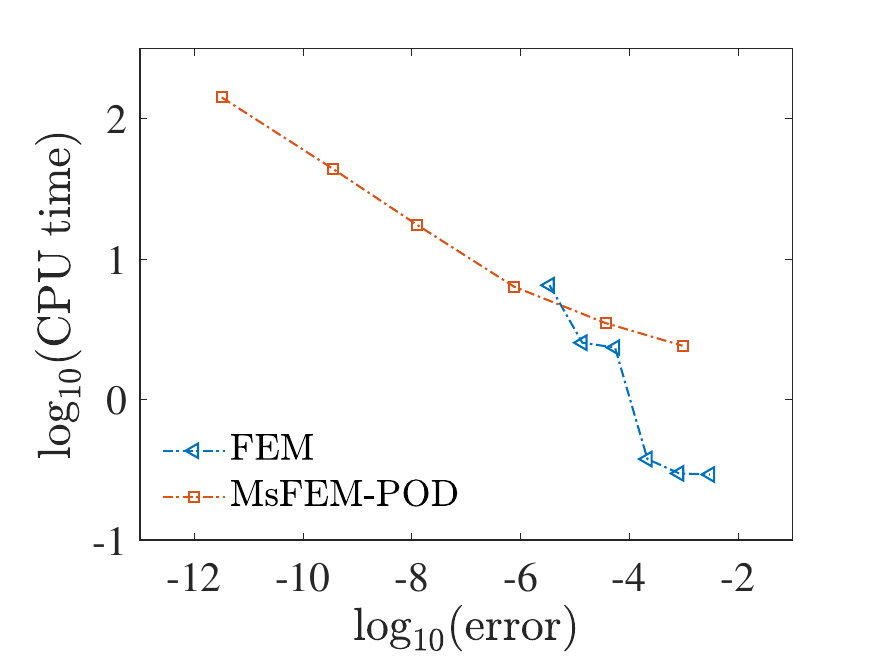}}
    \caption{Numerical convergence rates of FEM and MsFEM-POD in physic space and the comparison of the CPU time.}
    \label{fig:1D-convergence-CPU-time}
  \end{figure}

  At last, we compare the convergence rates of the qMC and MC methods. Both the FEM and MsFEM-POD are employed with the same computational setups. As shown in \Cref{fig:convergence-wrt-N}, the convergence rate of the qMC method reaches almost first-order in the random space.
  \begin{figure}[htbp]
    \centering
    \includegraphics[width=2.in]{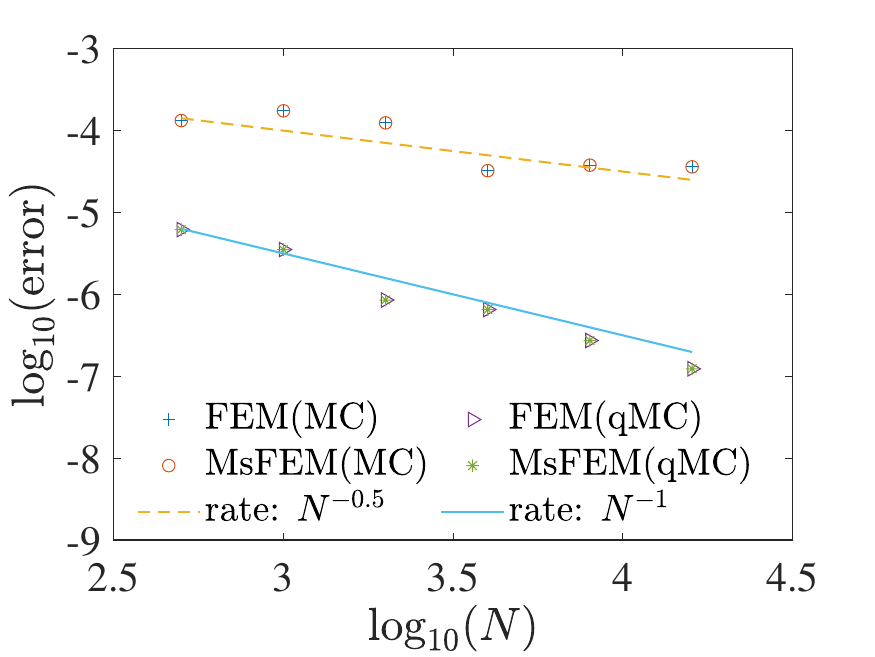}
    \caption{Numerical convergence rates of the FEM and the MsFEM-POD with respect to $N$. The "MsFEM" in the figure denotes the results provided by the MsFEM-POD method.}
    \label{fig:convergence-wrt-N}
  \end{figure}
\end{example}


\subsection{Localization of eigenfunctions}
At the end of this section, we employ the random potentials over the domain $[0, 1]^d$ ($d = 1,2$):
\begin{equation}
  V(x, \boldsymbol{\omega}_s) = v_0(x) + \sigma\sum_{j=1}^s\frac{1}{j^{q}}\sin(j\pi x)\omega_j,
  \label{equ:1d-potential-for-localization}
\end{equation}
where $\sigma$ denotes the strengthness of randomness. The 2D counterpart is
\begin{equation}
  V(\mathbf{x}, \boldsymbol{\omega}_s) = v_0(\mathbf{x}) + \sigma\sum_{j=1}^s\frac{1}{j^{q}}\sin(j\pi x)\sin(j\pi y)\omega_j.
  \label{equ:2d-potential-for-localization}
\end{equation}
Here we let $v_0(\mathbf{x})$ be a constant that ensures the minimal eigenvalues to be positive.
When $q \neq 0$, the high-frequency components of the potential are decaying with power rates. For $q = 0$, as $s \rightarrow \infty$, the potential converges to the spatially white noise. In the following experiments, we will check the reliability of the proposed method for both scenarios.

\begin{example}
  We set $q = 2$ and thus the amplitudes of high-frequency components are decaying very fast. Other parameters are $s = 32$, $h = 1/3200$, $\epsilon = 1.0$, and $N = 20000$.  For the MsFEM-POD method, we adopt $H = 1/10$ and compute all eigenvalues, while we compute the first 10 eigenvalues of the FEM approximated form. With 64 cores paralleling, the computational time of FEM is 341.17 seconds, while the MsFEM-POD method takes 29.72 seconds.
\begin{table}[htbp]
  \centering
  \caption{The comparison of mean and variance of first 5 eigenvalues computed by the FEM and MsFEM-POD methods.}
  \begin{tabular}{||c|c|c|c|c|c||}
    \hline
    & $\lambda_{1}$ & $\lambda_2$ & $\lambda_3$ & $\lambda_4$ & $\lambda_5$ \\
    \hline
    mean (FEM) & 0.9979 & 20.7075 & 20.7723 & 79.9490 & 79.9652 \\
    mean (MsFEM-POD) & 0.9979 & 20.7075 & 20.7724 & 79.9694 & 79.9856 \\
    error & 6.36e-08 & 5.31e-05 & 5.39e-05 & 2.04e-02 & 2.04e-02 \\
    \hline
    variance (FEM) & 0.1353 & 0.1359 & 0.1351 & 0.1353 & 0.1353 \\
    variance (MsFEM-POD) & 0.1353 & 0.1359 & 0.1351 & 0.1355 & 0.1354 \\
    error & 9.69e-11 & 1.49e-06 & 1.46e-06 & 1.32e-04 & 1.32e-04 \\
    \hline
  \end{tabular}
  \label{tab:1d-q2-mean-variance}
\end{table}

As illustrated in \Cref{tab:1d-q2-mean-variance}, the relative error of the mean between the FEM solution and the MsFEM-POD solution reaches an order of $10^{-4}$. Moreover, the MsFEM-POD method provides an extremely accurate solution for the minimal eigenvalue.
Besides, we record the means of the eigenvalues and the error as the $\epsilon$ varies. In \Cref{fig:1d_localization}, when we reduce the value of $\epsilon$, the accurate solution also can be produced by the MsFEM-POD method. This infers that for the random potential \eqref{equ:1d-potential-for-localization} with $q = 2$, the required dofs of the MsFEM-POD method are independent of the semiclassical parameter $\epsilon$.
\begin{figure}[htbp]
  \centering
  \includegraphics[width=2.in]{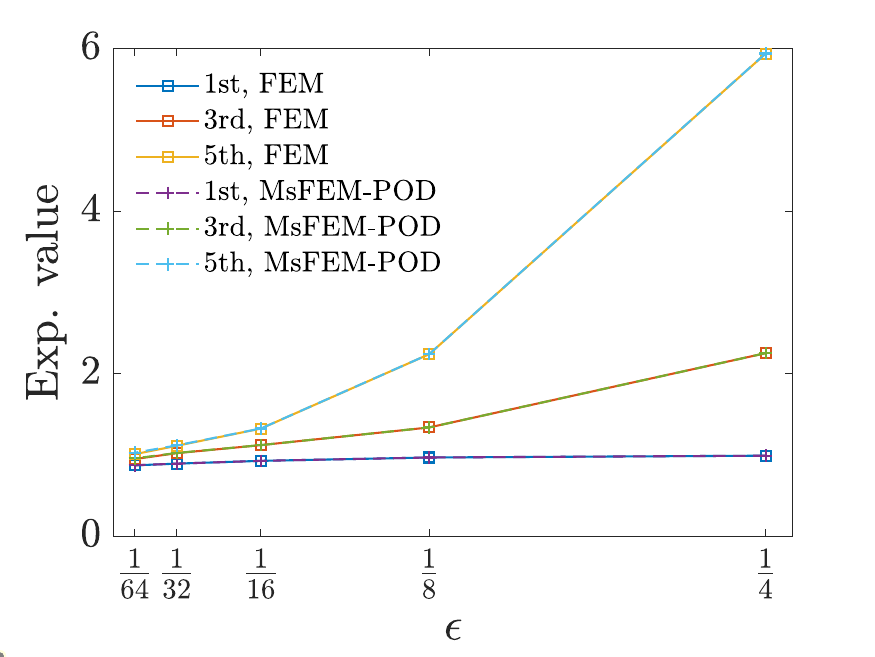}
  \includegraphics[width=2.in]{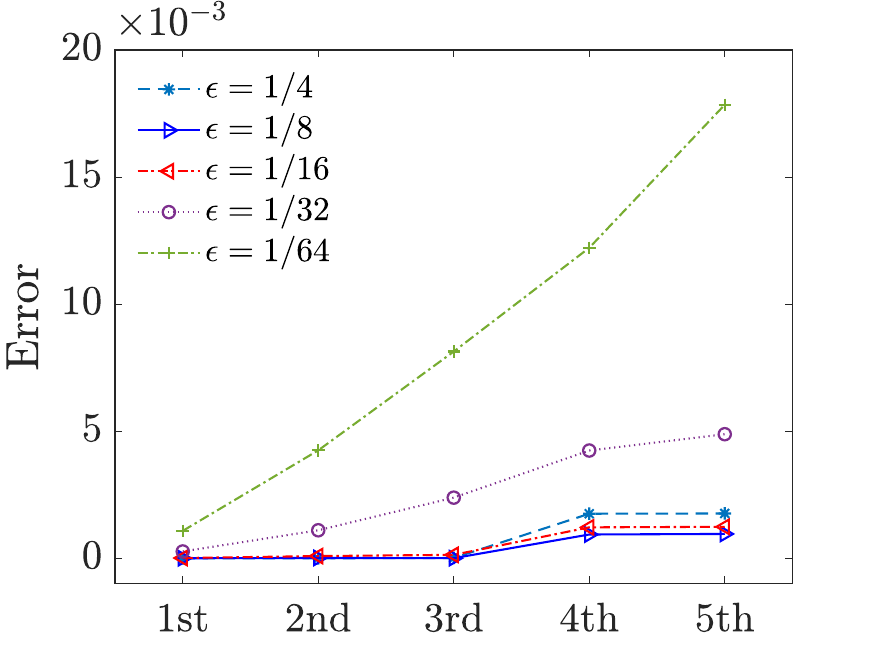}
  \caption{The first 5 eigenvalues computed by the FEM and MsFEM-POD for different semiclassical constant $\epsilon$.}
  \label{fig:1d_localization}
\end{figure}

We next consider the 2D case. We set $s = 32$, and the mesh size $h = \frac{1}{320}$ and $H = \frac{1}{10}$. The means and variances of the minimal eigenvalue as $\epsilon$ varies are recorded in~\Cref{tab:2d-mean-variance-eps-varies}. Meanwhile, we compare the ground states computed by the FEM and MsFEM-POD as shown in~\Cref{tab:2d-ground-states-eps-varies}. The numerical results indicate the effectiveness of our method.
\begin{table}[htbp]
  \centering
  \caption{The means and variances of the minimal eigenvalues computed by the FEM and MsFEM-POD methods for different $\epsilon$.}
  \begin{tabular}{||c|c|c|c|c|c||}
    \hline
    $\epsilon$ & $\frac 14$ & $\frac{1}{8}$ & $\frac{1}{16}$ & $\frac{1}{32}$ & $\frac{1}{64}$ \\
    \hline
    mean (FEM) & 0.9941 & 0.9777 & 0.9400  & 0.9031 & 0.8793 \\
    mean (MsFEM-POD) & 0.9941 & 0.9777 & 0.9400 & 0.9034 & 0.8807 \\
    error & 4.73e-06 & 5.25e-06 & 1.48e-05 & 2.73e-04 & 1.37e-03 \\
    \hline
    variance (FEM) & 1.38e-02 & 1.48e-02 & 1.89e-02 & 2.20e-02 & 2.34e-02 \\
    variance (MsFEM-POD) & 1.38e-02 & 1.48e-02 & 1.89e-02 & 2.20e-02 & 2.33e-02 \\
    error & 1.10e-06 & 1.77e-06 & 2.64e-06 & 2.00e-06 & 4.99e-05 \\
    \hline
  \end{tabular}
  \label{tab:2d-mean-variance-eps-varies}
\end{table}

\begin{figure}[htbp]
  \centering
  \subfloat[$\epsilon = \frac 14$.]{\includegraphics[width=1.6in]{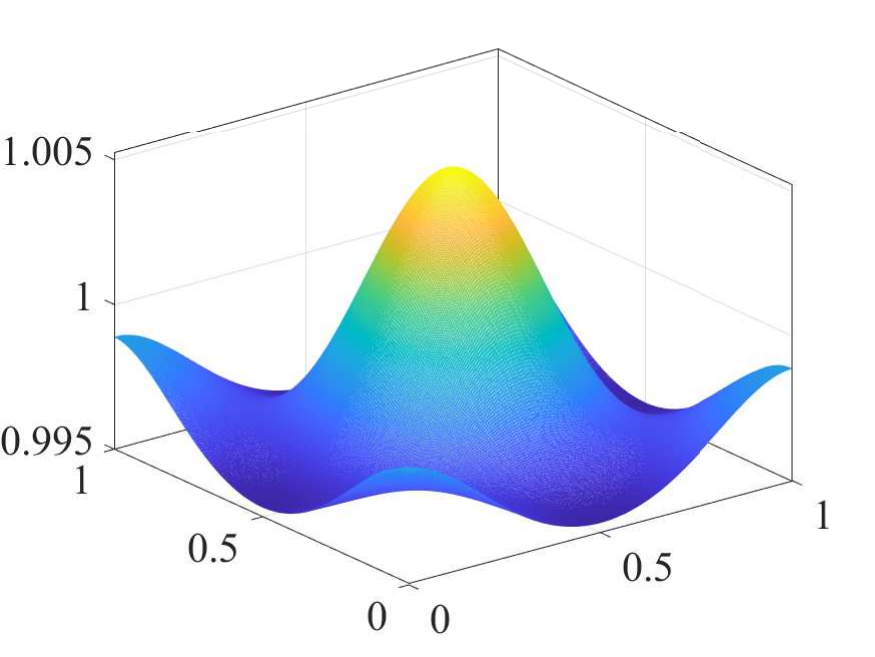}\includegraphics[width=1.6in]{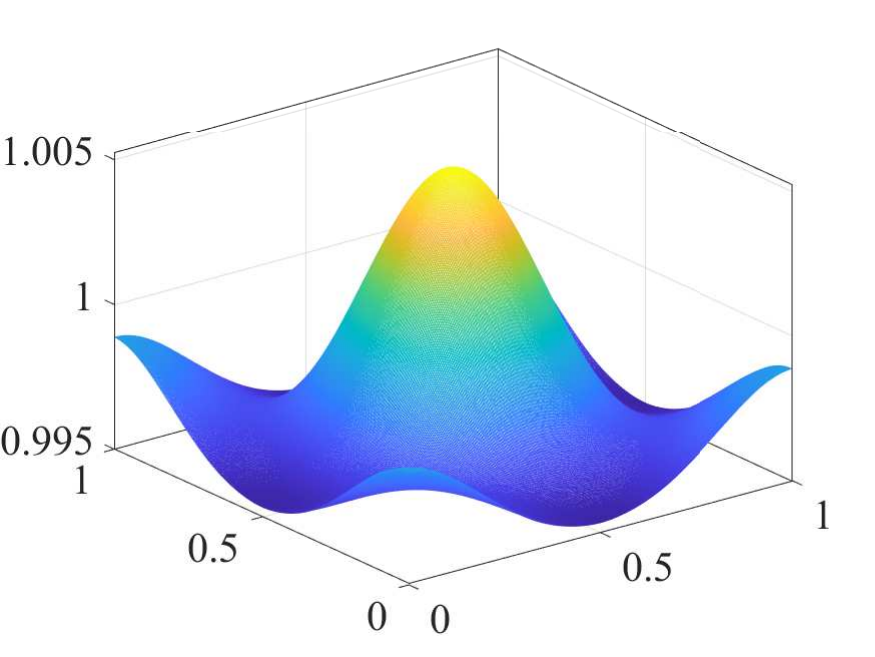}
  \includegraphics[width=1.6in]{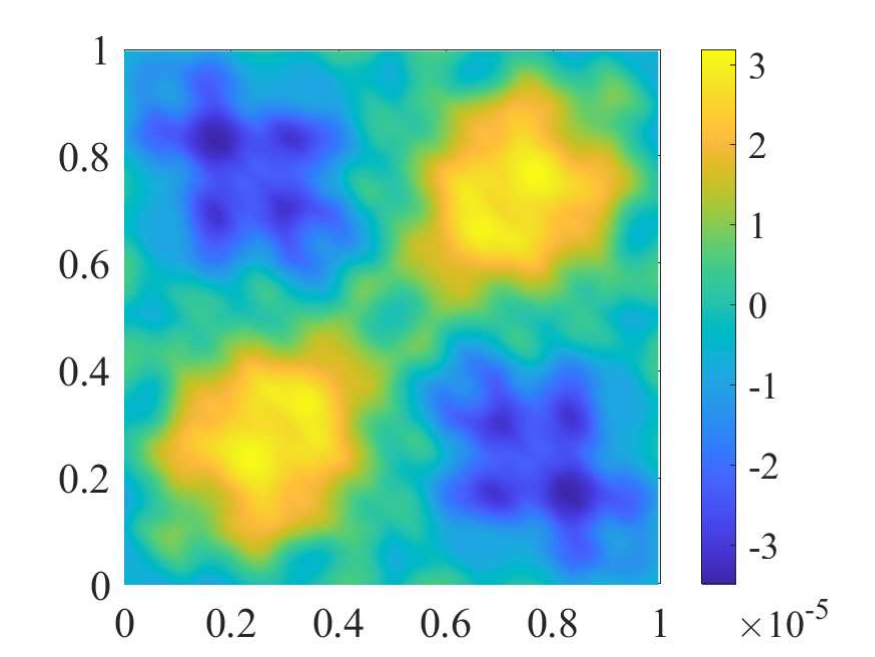}}\\
  \subfloat[$\epsilon = \frac {1}{64}$.]{\includegraphics[width=1.6in]{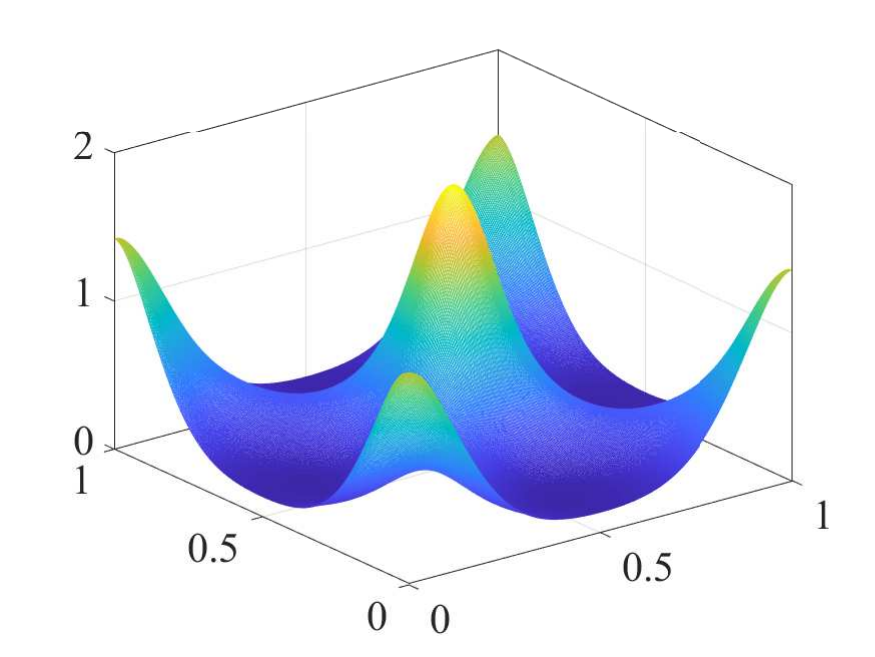}\includegraphics[width=1.6in]{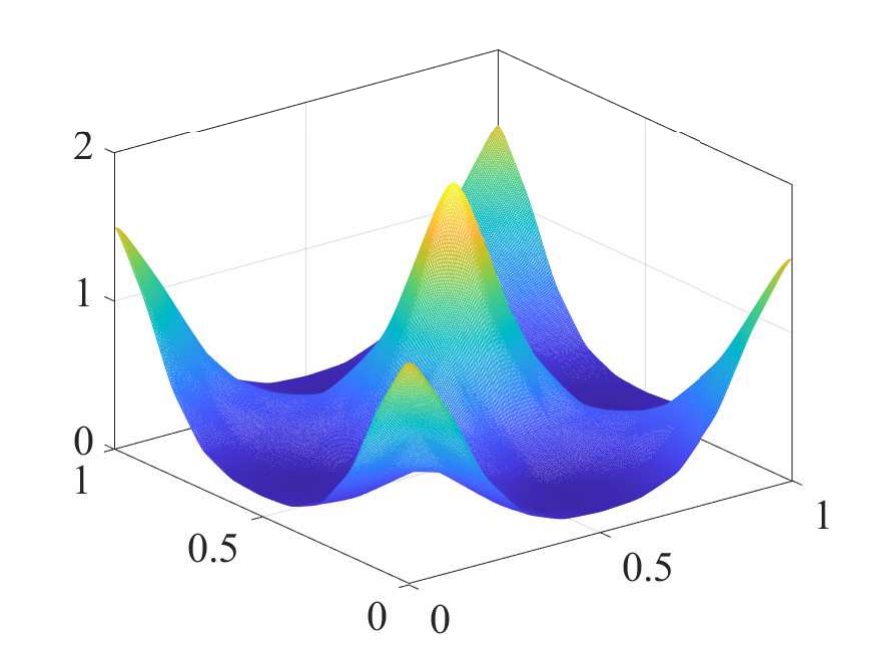}\includegraphics[width=1.6in]{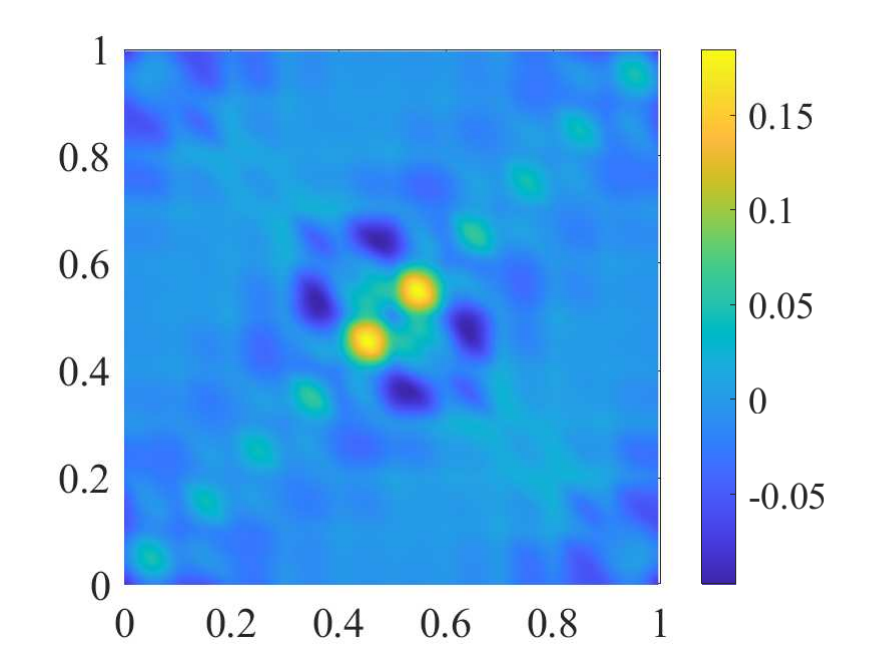}}
  \caption{The 2D ground states for different $\epsilon$. 1st column: FEM solution; 2nd column: MsFEM-POD solution; 3rd column: error distribution.}
  \label{tab:2d-ground-states-eps-varies}
\end{figure}
\end{example}

\begin{example}
  Here we consider $q = 0$ and simulate the spatial white noise. In this system, the localized eigenfunctions would be stabilized. For the 1D parameterized potential \eqref{equ:1d-potential-for-localization}, we fix $s = 256$, $\epsilon = \frac 1{16}$, and $h = \frac 1{15000}$. Numerical tests show that $H$ should be slightly smaller than $\epsilon$ but is independent of $s$. We set the coarse mesh $H = \frac 1{30}$ and obtain the localized eigenfunction as in \Cref{fig:1D-localization-eigenfunctions}.
  \begin{figure}[htbp]
    \centering
    \includegraphics[width=1.6in]{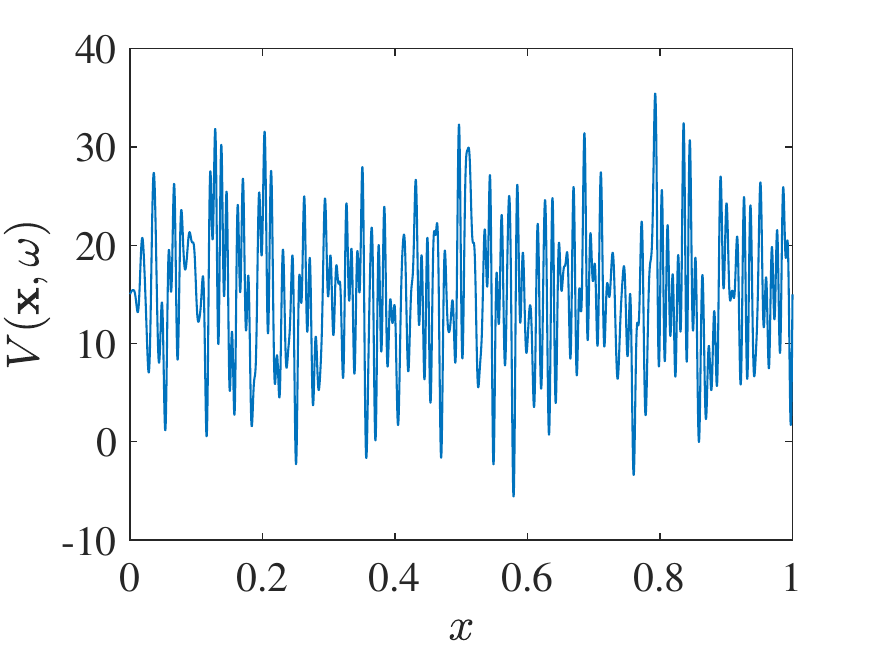}\includegraphics[width=1.6in]{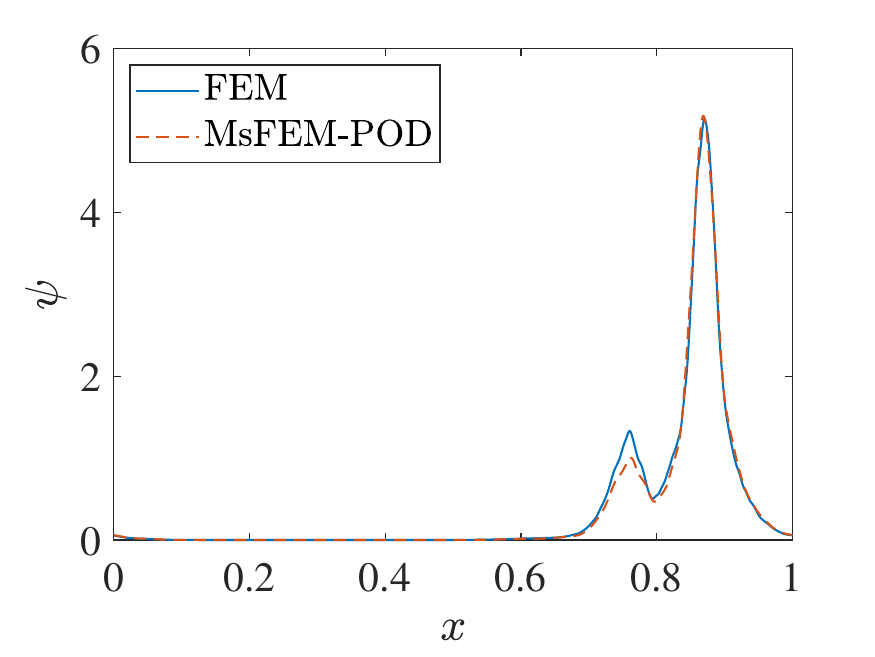}\includegraphics[width=1.6in]{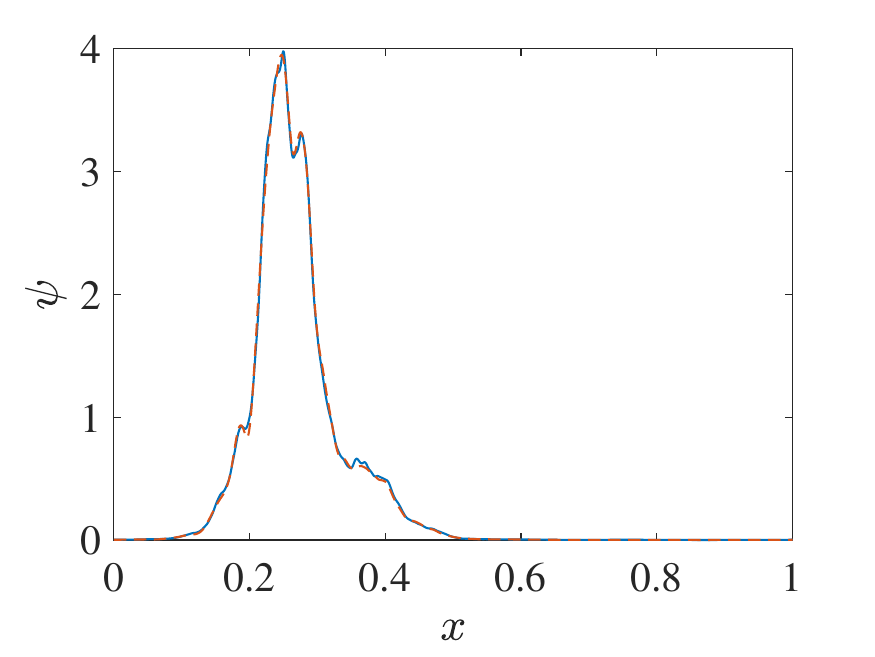}\\
    \includegraphics[width=1.6in]{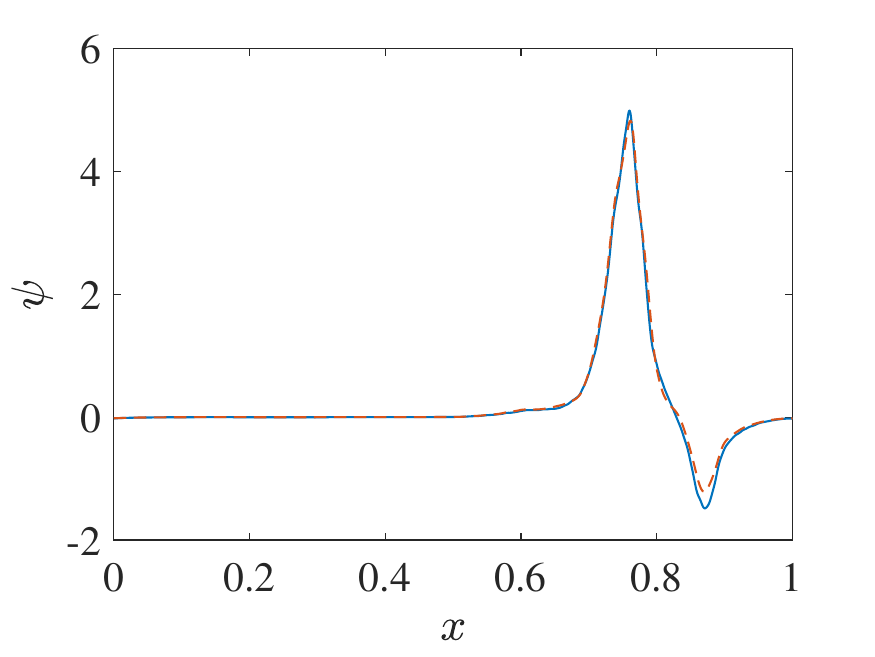}\includegraphics[width=1.6in]{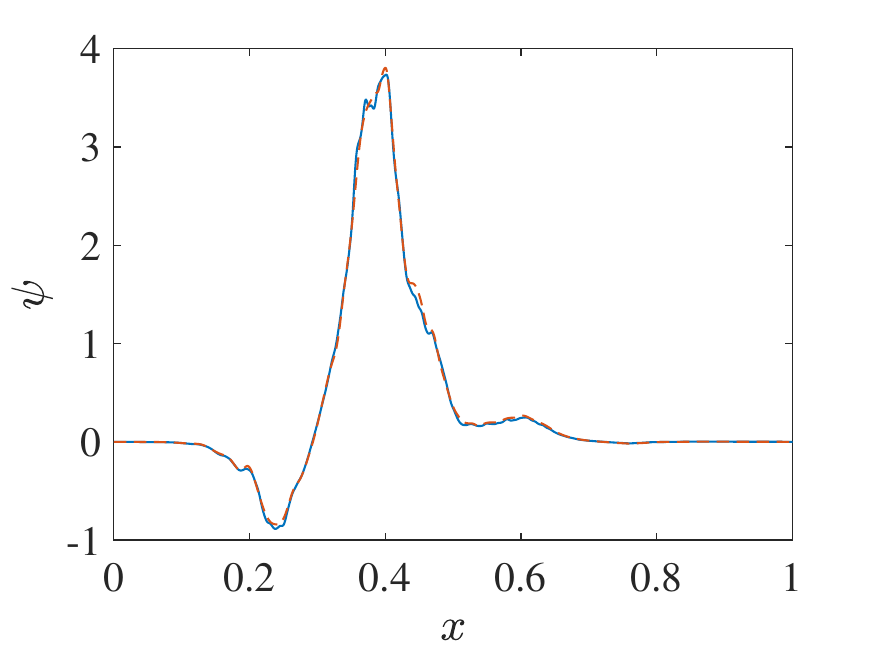}\includegraphics[width=1.6in]{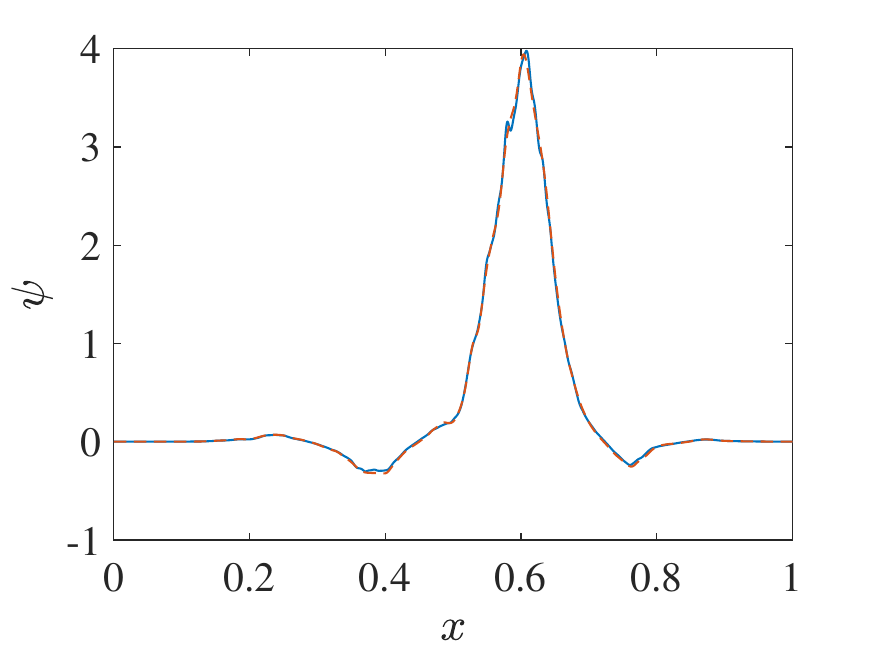}
    \caption{A realization of the random potential and the localized eigenfunctions corresponding to the first five minimal eigenvalues.}
    \label{fig:1D-localization-eigenfunctions}
  \end{figure}

Next, for the 2D problem, due to the memory limitation, we fix $s = 64$, and set $h = \frac 1{400}$ to ensure that the high-frequency features of the parameterized potential can be captured. The localization of the eigenfunctions is simulated with the coarse mesh size $H = \frac{1}{20}$ as in \Cref{fig:2d-localization-random-potential-q-0}. Here the results computed by the FEM are not depicted, but we depict the first five eigenvalues to demonstrate the reliability of the MsFEM-POD method as in \Cref{tab:2d-eigenvalues_error}.
\begin{figure}[htbp]
  \centering
  \includegraphics[width=1.6in]{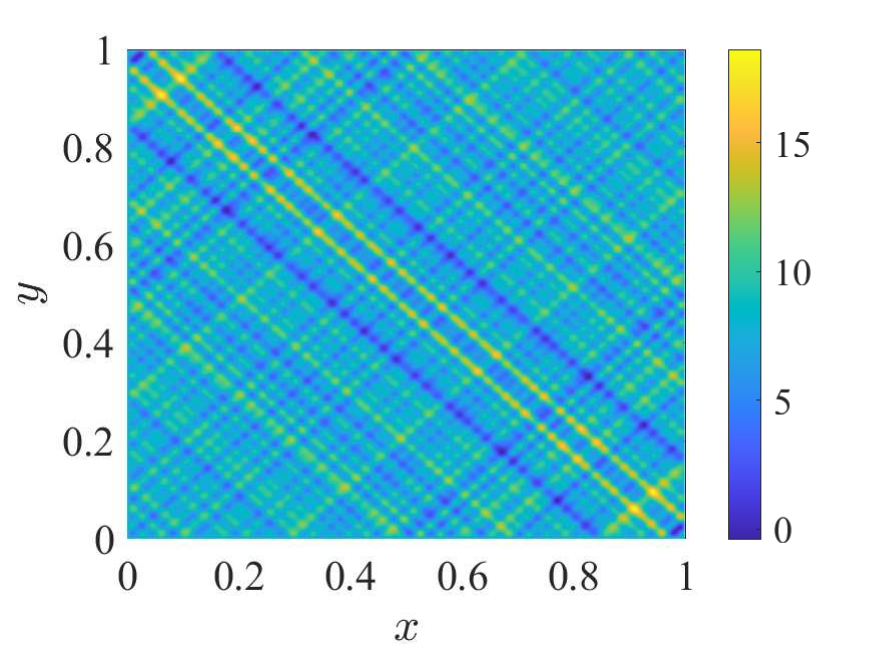}\includegraphics[width=1.6in]{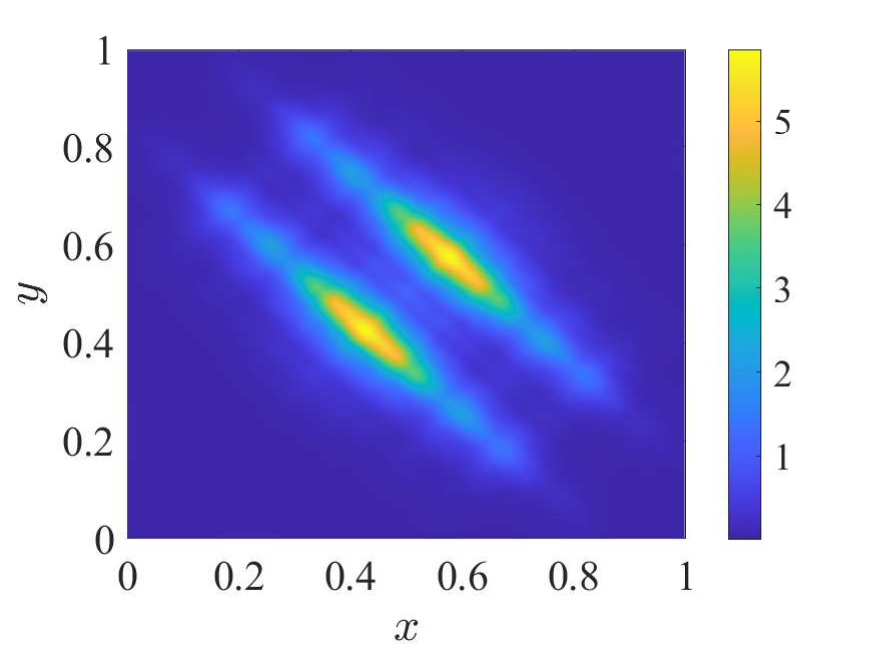}\includegraphics[width=1.6in]{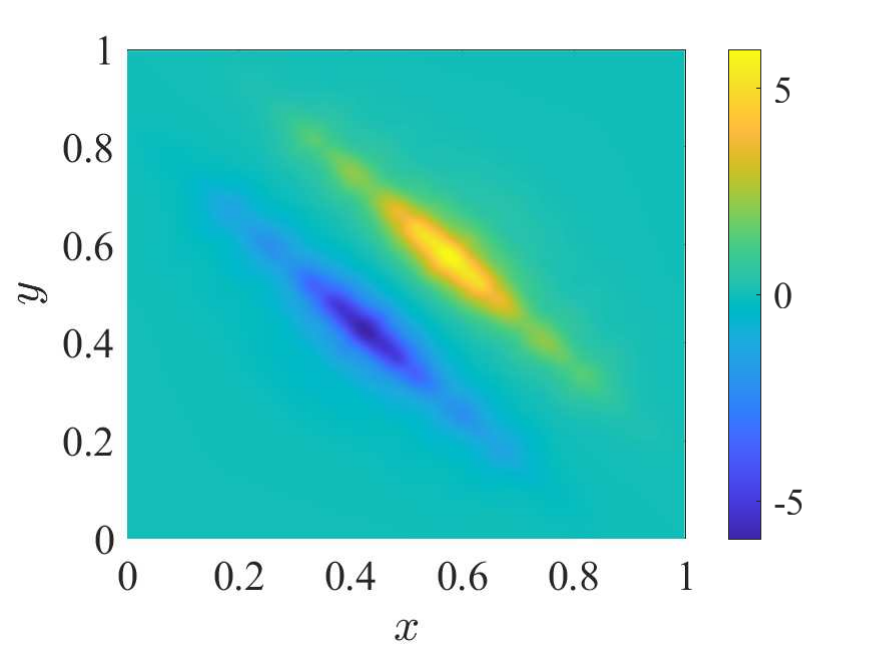}\\
  \includegraphics[width=1.6in]{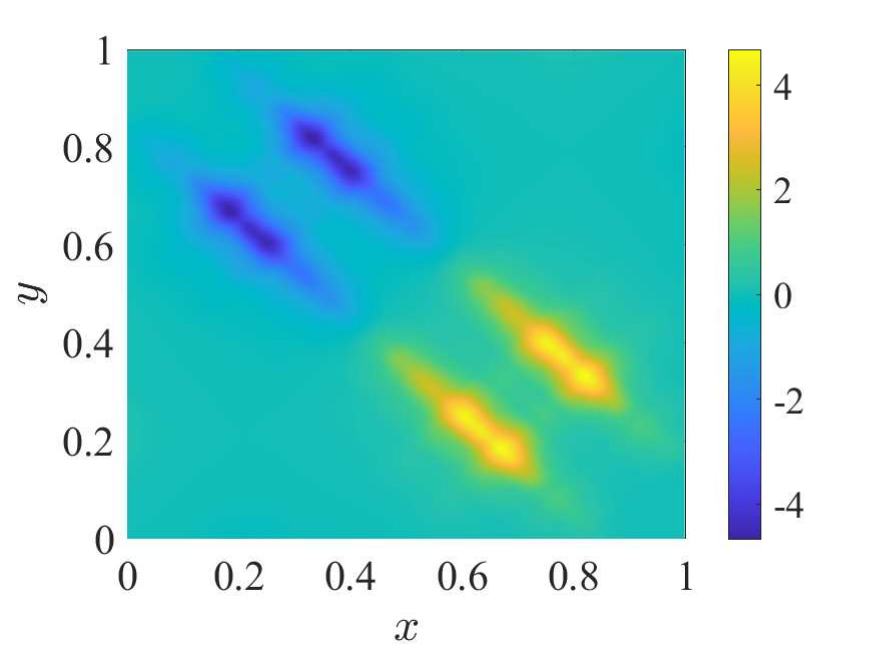}\includegraphics[width=1.6in]{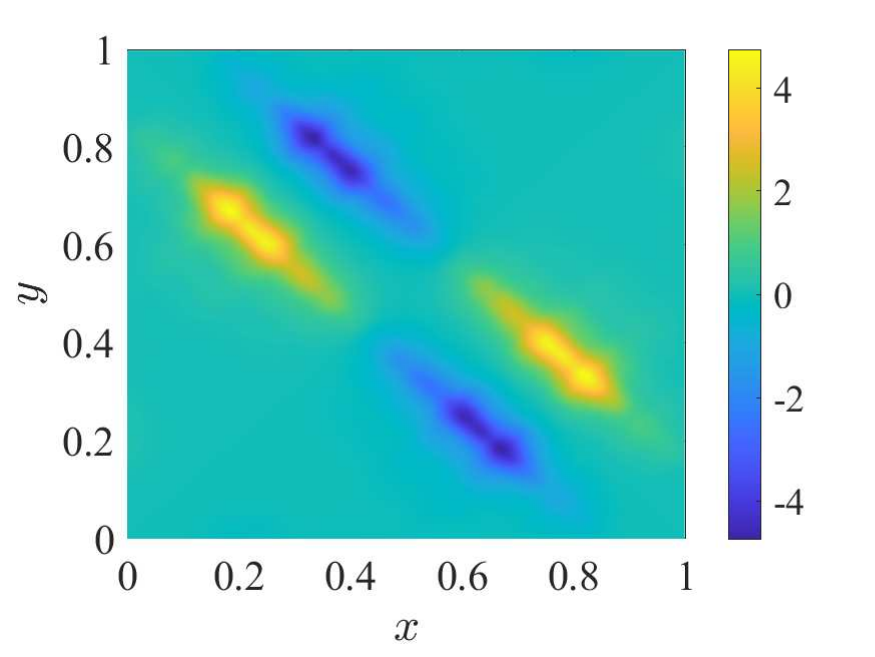}\includegraphics[width=1.6in]{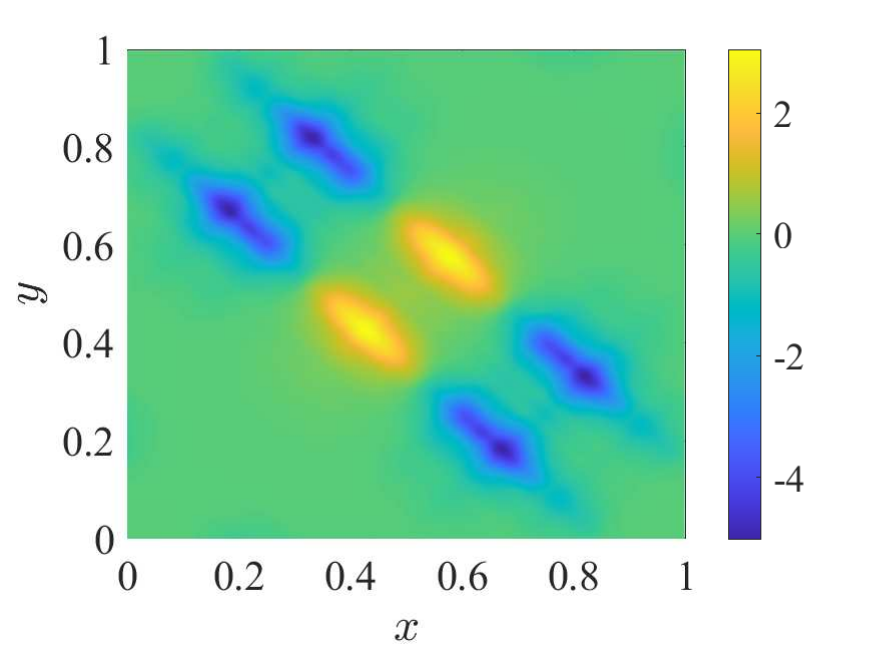}
  \caption{A realization of the 2D parameterized random potential and the localized eigenfunctions computed by the MsFEM-POD method. The corresponding eigenvalues are shown in~\Cref{tab:2d-eigenvalues_error}.}
  \label{fig:2d-localization-random-potential-q-0}
\end{figure}
\begin{table}[htbp]
  \centering
  \caption{The first five eigenvalues computed by the FEM and MsFEM-POD methods.}
  \begin{tabular}{|c|ccccc|}
    \hline
    FEM & 6.7399 & 6.7636 & 6.8224 & 6.8461 & 6.8542 \\
    MsFEM-POD & 6.7819 & 6.8055 & 6.8779 & 6.9018 & 6.9144 \\
    \hline
    absolute error & 4.1999e-02 & 4.1919e-02 & 5.5432e-02 & 5.5686e-02 & 6.0148e-02 \\
    relative error & 6.2314e-03 & 6.1977e-03 & 8.1249e-03 & 8.1340e-03 & 8.7753e-03 \\
    \hline
  \end{tabular}
  \label{tab:2d-eigenvalues_error}
\end{table}
\end{example}

\begin{remark}
    When we set $q = 0$ in the parameterized random potentials \eqref{equ:1d-potential-for-localization} and \eqref{equ:2d-potential-for-localization}, the bounds of the random potentials directly depend on the truncated dimension. For this class of problems, the conditions outlined in \Cref{assump:assumption-for-potentials}(2) and (3) cannot be sustained, resulting in the lack of convergent eigenvalues and eigenfunctions. Nevertheless, when the condition \Cref{assump:assumption-for-potentials}(1) is satisfied, i.e. $H < \epsilon$, the localization of eigenfunctions is simulated accurately with lower computational cost, which demonstrates the application potential of the proposed MsFEM-POD method on simulating complex quantum systems governed by semiclassical random Schr\"odinger operators.
\end{remark}

\section{Conclusions}
\label{sec:conclusions}
In this paper, we present a multiscale reduced method for solving the eigenvalue problem of the semiclassical random Schr\"odinger operators. The random potential of the Schr\"odinger operator is parameterized by a truncated series with random parameters. We introduce the MsFEM to approximate the resulting problem, where the order-reduced multiscale basis functions are constructed by an effective approach based on the POD method. Moreover, we prove an error estimate for the proposed method, where the approximation error is a combined form consisting of the model truncation error, the MsFEM approximation error, the POD truncation error, and the integral approximation error of the quasi-Monte Carlo method. We also conduct numerical experiments to validate the error estimate. Using the proposed method, we accurately simulate the Anderson localization of eigenfunctions for spatially random potentials. The results demonstrate that our approach provides a practical and efficient solution for simulating complex quantum systems governed by semiclassical random Schr\"odinger operators.

\section*{Acknowledgements}
The research of Z. Zhang is supported by the National Natural Science Foundation of China (projects 92470103 and 12171406), the Hong Kong RGC grant (projects 17307921 and 17304324), the Outstanding Young Researcher Award of HKU (2020-21), and seed funding from the HKU-TCL Joint Research Center for Artificial Intelligence. The simulations are performed using research computing facilities offered by Information Technology Services, University of Hong Kong.

\section*{Declaration of interest}
The authors report no conflict of interest.

\bibliographystyle{amsplain}
\bibliography{references}

\end{document}